\documentclass[a4paper,12pt]{amsart}
\usepackage{amssymb}
\usepackage{url}
\usepackage[T1]{fontenc}

\input xypic
\xyoption{all}
\xyoption{arc}

\newtheorem{theo}{Theorem}[section]

\newtheorem{prop}[theo]{Proposition}

\newtheorem{lem}[theo]{Lemma}

\newtheorem{coro}[theo]{Corollary}

\def\remark#1{{\refstepcounter{theo}\label{#1}\noindent\sc Remark  
\arabic{section}.\arabic{theo} - }}

\def\equat{\refstepcounter{theo}$$~}
\def\endequat{\leqno{\boldsymbol{(\arabic{section}.\arabic{theo})}}~$$}

    \def\NM{{\mathbb{N}}}

\def\SG{{\mathfrak S}}

    \def\ZM{{\mathbb{Z}}}

  \def\ab{{\mathbf a}}  
  \def\bb{{\mathbf b}}  
    
    \def\DC{{\mathcal{D}}}
    \def\EC{{\mathcal{E}}}

  \def\hb{{\mathbf h}}  \def\HC{{\mathcal{H}}}

    \def\PC{{\mathcal{P}}}
    
    \def\RC{{\mathcal{R}}}
  \def\sb{{\mathbf s}}

    \def\XC{{\mathcal{X}}}

\def\a{\alpha}
\def\b{\beta}
\def\g{\gamma}
\def\G{\Gamma}
\def\d{\delta}
\def\D{\Delta}

\def\ph{\varphi}
\def\l{\lambda}

\def\m{\mu}

\def\s{\sigma}

\def\t{\tau}

            \def\mut{{\tilde{\mu}}}

\DeclareMathOperator{\Id}{{\mathrm{Id}}}

\def\to{\rightarrow}
\def\longto{\longrightarrow}

\def\fonction#1#2#3#4#5{\begin{array}{rccc}
{#1} : & {#2} & \longto & {#3} \\
& {#4} & \longmapsto & {#5} 
\end{array}}

\def\fonctio#1#2#3#4{\begin{array}{ccc}
{#1} & \longto & {#2} \\
{#3} & \longmapsto & {#4} 
\end{array}}

\def\vide{\varnothing}
\def\DS{\displaystyle}
\def\SS{\scriptstyle}

\def\finl{~$\SS \square$}

\def\lexp#1#2{\kern\scriptspace\vphantom{#2}^{#1}\kern-\scriptspace#2}
\def\le{\hspace{0.1em}\mathop{\leqslant}\nolimits\hspace{0.1em}}
\def\ge{\hspace{0.1em}\mathop{\geqslant}\nolimits\hspace{0.1em}}

\mathchardef\inferieur="321E
\mathchardef\superieur="321F

\def\eqna{\begin{eqnarray*}}
\def\endeqna{\end{eqnarray*}}

\def\itemth#1{\item[${\mathrm{(#1)}}$]}

\catcode`\@=11
\long\def\@car#1#2\@nil{#1}
\long\def\@first#1#2{#1}
\long\def\@second#1#2{#2}
\long\def\ifempty#1{\expandafter\ifx\@car#1@\@nil @\@empty
  \expandafter\@first\else\expandafter\@second\fi}
\catcode`\@=12

\newcounter{soussection}[section]

\def\soussection#1{\refstepcounter{soussection}
    \noindent{\bf \arabic{section}.\Alph{soussection}. #1.}
    \addcontentsline{toc}{section}{\quad 
        {\arabic{section}.\Alph{soussection}. #1.}}}

\def\gauche{\stackrel{L}{\longleftarrow}}
\def\droite{\stackrel{R}{\longleftarrow}}
\def\bilatere{\stackrel{LR}{\longleftarrow}}
\def\pre#1{\leqslant_{#1}}
\def\shape{\sb\hb}

\begin{document}

\baselineskip=16pt

\title{On Kazhdan-Lusztig cells in type ${\boldsymbol{B}}$}

\author{C\'edric Bonnaf\'e}
\address{\noindent 
Labo. de Math. de Besan\c{c}on (CNRS: UMR 6623), 
Universit\'e de Franche-Comt\'e, 16 Route de Gray, 25030 Besan\c{c}on
Cedex, France} 

\makeatletter
\email{cedric.bonnafe@univ-fcomte.fr}

\thanks{The author is partly supported by the ANR (Project 
No JC07-192339)}

\makeatother

\subjclass{According to the 2000 classification:
Primary 20C08; Secondary 20C15}

\date{\today}

\begin{abstract} 
We prove that, for any choice of parameters, the 
Kazhdan-Lusztig cells of a Weyl group of type $B$ 
are unions of combinatorial cells 
(defined using the domino insertion algorithm). 
\end{abstract}

\maketitle

\pagestyle{myheadings}

\markboth{\sc C. Bonnaf\'e}{\sc On Kazhdan-Lusztig cells in type $B$}

%\bigskip

Let $(W_n,S_n)$ be the Weyl group of type $B_n$, where
$S_n=\{t,s_1,\dots,s_{n-1}\}$ and where 
the Dynkin diagram is given by
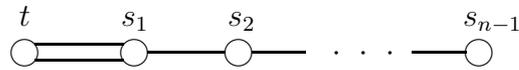
\begin{center}
\begin{picture}(220,30)
\put( 40, 10){\circle{10}}
\put( 44,  7){\line(1,0){33}}
\put( 44, 13){\line(1,0){33}}
\put( 81, 10){\circle{10}}
\put( 86, 10){\line(1,0){29}}
\put(120, 10){\circle{10}}
\put(125, 10){\line(1,0){20}}
\put(155,  7){$\cdot$}
\put(165,  7){$\cdot$}
\put(175,  7){$\cdot$}
\put(185, 10){\line(1,0){20}}
\put(210, 10){\circle{10}}
\put( 38, 20){$t$}
\put( 76, 20){$s_1$}
\put(116, 20){$s_2$}
\put(204, 20){$s_{n{-}1}$}
\end{picture}
\end{center}
Let $\ell : W_n \to \NM=\{0,1,2,3,\dots\}$ be the length function. 
Let $\G$ be a totally ordered abelian group and let 
$\ph : W_n \to \G$ be a weight function (in the sense 
of Lusztig \cite[\S 3.1]{lusztig}). We set
$$\ph(t)=b\quad\text{and}\quad \ph(s_1)=\cdots=\ph(s_{n-1})=a.$$
To this datum, the Kazhdan-Lusztig theory (with unequal parameters 
\cite{lusztig}) associates a partition of $W_n$ into 
left, right or two-sided cells \cite[Chapter 8]{lusztig}.

In \cite[Conjectures A and B]{bgil}, 
Geck, Iancu, Lam and the author have proposed several 
conjectures for describing these partitions (at least whenever 
$a$, $b > 0$, but this is not such a big restriction, as can be 
seen from \cite[Corollary 5.8]{semicontinu}): they involve a {\it domino 
insertion algorithm}. Roughly speaking, one can define a 
partition of $W_n$ into {\it combinatorial $($left, right or two-sided$)$ 
$(a,b)$-cells} (which depend on $a$, $b$ and which are defined 
combinatorially using the domino insertion algorithm): the 
combinatorial (left, right or two-sided) cells should coincide 
with the Kazhdan-Lusztig (left, right or two-sided) cells.
The aim of this paper is to prove one of the two inclusions 
(see Theorem \ref{main}):

\bigskip

\noindent{\bf Theorem.} 
{\it If two elements of $W_n$ lie in the same combinatorial 
(left, right or two-sided) cell, then they lie in the same Kazhdan-Lusztig 
(left, right or two-sided) cell.}

\bigskip

In the case of the symmetric group, the partition into left 
cells (obtained by Kazhdan and Lusztig \cite[Theorem 1.4]{KL}) 
uses the Robinson-Schensted correspondence, and the key tool is 
a description of this correspondence using 
plactic/coplactic relations (also called 
Knuth relations). For $W_n$, whenever $b > (n-1) a$, the 
partition into left, right or two-sided cells 
was obtained by Iancu and the author 
(see \cite[Theorem 7.7]{lacriced} and 
\cite[Corollaries 3.6 and 5.2]{bilatere}) 
again by using the translation of a generalised Robinson-Schensted 
correspondence through plactic/coplactic relations. 

Recently, M. Taskin \cite{taskin} and T. Pietraho \cite{pietraho} 
have independently provided plactic/coplactic 
relations for the domino insertion algorithm. Our methods 
rely heavily on their results: we show that, if two elements 
of $W_n$ are directly related by a plactic relation, then 
they are in the same Kazhdan-Lusztig cell. The key step will be the 
Propositions \ref{exemples m1} and \ref{exemples m2}, where 
some multiplications between elements of the Kazhdan-Lusztig 
basis are computed by brute force. We then derive some  
consequences (see Propositions \ref{quasi asymptotique} and 
\ref{coro exemples m2}), where it is proved that some elements 
are in the same left cells. Then, the rest of the proof just 
uses the particular combinatoric of Weyl groups of type $B$, 
together with classical properties of Kazhdan-Lusztig cells. 
%These computations can be carried out using ideas coming from Geck's paper 
%on the induction of Kazhdan-Lusztig cells \cite{geck induction}.
\bigskip

\noindent{\bf Acknowledgements.} Part of this work was done while the 
author stayed at the MSRI during the winter 2008. The author wishes to 
thank the Institute for its hospitality and the organizers of the two 
programs of that period for their invitation.

The author also wants to thank warmly L. Iancu and N. Jacon for their careful reading 
of a preliminary version of this paper and for their useful comments.

\bigskip

\tableofcontents

\section{Notation}

\medskip

\soussection{Weyl group}
Let $(W_n,S_n)$ be the Weyl group of type $B_n$, where
$S_n=\{t,s_1,\dots,s_{n-1}\}$ and where 
the Dynkin diagram is given by
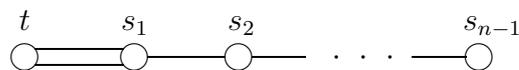
\begin{center}
\begin{picture}(220,30)
\put( 40, 10){\circle{10}}
\put( 44,  7){\line(1,0){33}}
\put( 44, 13){\line(1,0){33}}
\put( 81, 10){\circle{10}}
\put( 86, 10){\line(1,0){29}}
\put(120, 10){\circle{10}}
\put(125, 10){\line(1,0){20}}
\put(155,  7){$\cdot$}
\put(165,  7){$\cdot$}
\put(175,  7){$\cdot$}
\put(185, 10){\line(1,0){20}}
\put(210, 10){\circle{10}}
\put( 38, 20){$t$}
\put( 76, 20){$s_1$}
\put(116, 20){$s_2$}
\put(204, 20){$s_{n{-}1}$}
\end{picture}
\end{center}
Let $\ell : W_n \to \NM=\{0,1,2,3,\dots\}$ be the length function. 
Let $I_n=\{\pm1,\dots,\pm n\}$: we shall identify $W_n$ with the 
group of permutations $w$ of $I_n$ such that $w(-i)=-w(i)$ 
for all $w \in I_n$. The identification is through the following map
$$t \longmapsto (1,-1)\quad\text{and}\quad s_i 
\longmapsto (i,i+1)(-i,-i-1).$$

The next notation comes from \cite[\S 4]{lacriced}: it is rather 
technical but will be used throughout this paper. 
We set $t_1=r_1=t$ and, for $1 \le i \le n-1$, we set
$$r_{i+1}=s_i r_i\quad\text{and}\quad t_{i+1}=s_it_is_i.$$
We shall often use the following well-known lemma:

\bigskip

\begin{lem}\label{longueur}
Let $w \in W_n$, $i \in \{1,2,\dots,n-1\}$ and $j \in \{1,2,\dots,n\}$. 
Then:
\begin{itemize}
\itemth{a} $\ell(ws_i) > \ell(w)$ (that is, $ws_i > w$) if and only 
if $w(i) < w(i+1)$.

\itemth{b} $\ell(wt_j) > \ell(w)$ if and only if $w(j) > 0$.
\end{itemize}
\end{lem}

\bigskip

As a permutation of $I_n$, we have
\equat\label{ti}
t_i=(i,-i)
\endequat
and
\equat\label{ri}
r_i(j)=
\begin{cases}
-i & \text{if $j=1$,}\\
j-1 & \text{if $2 \le j \le i$,}\\
j & \text{if $i+1 \le j \le n$.}
\end{cases}
\endequat
An easy computation shows that, if $j \in \{1,2,\dots,n-1\}$ 
and $i \in \{1,2,\dots,n\}$, then
\equat\label{sk rl}
s_j r_i =
\begin{cases}
r_i s_j & \text{if $j > i$,}\\
r_{i+1} & \text{if $j=i$,}\\
r_{i-1} & \text{if $j=i-1$,}\\
r_i s_{j+1} & \text{if $1 \le j < i-1$.}
\end{cases}
\endequat
Note also that, if $l \ge 2$, then 
\equat\label{rl rl}
r_l r_l = r_{l-1} r_l s_1.
\endequat
We set $a_0=1$ and, if $0 \le l \le n$, we set
$$a_l=r_1 r_2 \cdots r_l.$$
As a permutation of $I_n$, we have 
\equat\label{al}
a_l(i)=
\begin{cases}
i-1-l & \text{if $1 \le i \le l$,}\\
i & \text{if $l+1 \le i \le n$.}
\end{cases}
\endequat
In particular,
\equat\label{inverse al}
a_l^{-1}=a_l
\endequat
and, if $i \in \{1,2,\dots,n-1\}\setminus\{l\}$, then
\equat\label{al s}
a_l s_i a_l =
\begin{cases}
s_{l-i} & \text{if $i < l$},\\
s_i & \text{if $i > l$.}
\end{cases}
\endequat
Note also that
\equat\label{l al}
\ell(a_l)=\frac{l(l+1)}{2}.
\endequat

We shall identify the symmetric group $\SG_n$ with the subgroup 
of $W_n$ generated by $s_1$,\dots, $s_{n-1}$. We also 
set $I_n^+=\{1,2,\dots,n\}$. Then, as a group of permutations 
of $I_n$, we have
\equat\label{sn}
\SG_n=\{w \in W_n~|~w(I_n^+)=I_n^+\}.
\endequat
If $1 \le i \le j \le n$, we denote by $[i,j]$ the set 
$\{i,i+1,\dots,j\}$ and by $\SG_{[i,j]}$ the subgroup 
of $W_n$ (or of $\SG_n$) generated by $s_i$, $s_{i+1}$,\dots, $s_{j-1}$. 
If $j < i$, then we set $[i,j] =\vide$ and $\s_{[i,j]}=1$. 
As a group of permutations of $I_n$, we have
\equat\label{sij}
\SG_{[i,j]}=\{w \in \SG_n~|~
\forall k \in I_n^+\setminus [i,j], w(k)=k\}.
\endequat

The longest element of $W_n$ will be denoted by $w_n$ 
(it is usually denoted by $w_0$, but since we shall use 
induction on $n$, we need to emphasize its dependence on $n$). 
We denote by $\s_n$ the longest element of $\SG_n$. The longest 
element of $\SG_{[i,j]}$ will be denoted by $\s_{[i,j]}$. 
As a permutation of $I_n$, we have
\equat\label{wn}
w_n = (1,-1)(2,-2)\cdots (n,-n).
\endequat
Note also that
\equat\label{wn an sn}
\begin{cases}
w_n & = t_1 t_2 \cdots t_n =t_n\cdots t_2 t_1 \\
w_n & = a_n \s_n = \s_n a_n,\\
\s_n & = \s_{[1,n]} 
\end{cases}
\endequat
and that
\equat\label{centre}
\textit{$w_n$ is central in $W_n$.}
\endequat

\bigskip

\soussection{Decomposition of elements of ${\boldsymbol{W_n}}$} 
If $0 \le l \le n$, we denote by $\SG_{l,n-l}$ the subgroup 
of $\SG_n$ generated by $\{s_1,\dots,s_{n-1}\}\setminus\{s_l\}$. Then 
$\SG_{l,n-l}=\SG_{[1,l]} \times \SG_{[l+1,n]} \simeq \SG_l \times \SG_{n-l}$. 
We denote by $Y_{l,n-l}$ the set of elements $w \in \SG_n$ 
which are of minimal length in $w\SG_{l,n-l}$. Note that $a_l$ normalizes $\SG_{l,n-l}$ 
(this follows from \ref{al s}).

If $w \in W_n$, we denote by $\ell_t(w)$ the number of occurences 
of $t$ in a reduced decomposition of $w$ (this does not depend on 
the choice of the reduced decomposition). We set 
$\ell_s(w)=\ell(w)-\ell_t(w)$.

\bigskip

\begin{lem}\label{decomposition}
Let $w \in W_n$. Then there exists a unique quadruple 
$(l,\a,\b,\s)$ where $0 \le l \le n$, $\a$, $\b \in Y_{l,n-l}$ and 
$\s \in \SG_{l,n-l}$ are such that $w=\a a_l\s \b^{-1}$. Moreover, 
there exists a unique sequence $1 \le i_1 < i_2 < \cdots < i_l \le n$ 
such that $\a a_l = r_{i_1} r_{i_2}\cdots r_{i_l}$. 
We have
$$\ell(w)=\ell(\a)+\ell(a_l)+\ell(\s)+\ell(\b),$$
$$\ell_t(w)=l$$
$$\{i_1,\dots,i_l\}=\{i \in [1,n]~|~w^{-1}(i) < 0\}.\leqno{\textit{and}}$$
Note also that 
$$\ell(\a)=\sum_{k=1}^l (i_k-k).$$
\end{lem}

\bigskip

\begin{proof}
See \cite[\S 4, and especially Proposition 4.10]{lacriced}.
\end{proof}

\bigskip

If $l \in [0,n]$ and if $1 \le i_1 < \cdots < i_l \le n$ and 
$1 \le j_1 < \cdots < j_{n-l} \le n$ are two sequences 
such that $[1,n]=\{i_1,\dots,i_l\} \cup \{j_1,\dots,j_{n-l}\}$, then 
it follows easily from \ref{ri} that 
\equat\label{ri1 ril}
\begin{cases}
(r_{i_1}\cdots r_{i_l})^{-1}(i_k)=k-l-1 & \text{if $1 \le k \le l$,}\\
(r_{i_1}\cdots r_{i_l})^{-1}(j_k)=l+k & \text{if $1 \le k \le n-l$.}\\
\end{cases}
\endequat

\bigskip

The elements $\a$, $\b$ and $\s$ of the previous lemma will we denoted by 
$\a_w$, $\b_w$ and $\s_w$ respectively. We have
\equat\label{inverse}
\ell_t(w^{-1})=\ell_t(w),\quad 
\a_{w^{-1}}=\b_w,\quad \b_{w^{-1}}=\a_w\quad \text{and}\quad
\s_{w^{-1}}=a_l(\s_w)^{-1}a_l.
\endequat

We shall now describe how the multiplication by the longest 
element $w_n$ acts on the decomposition given by Lemma \ref{decomposition}. 
For this, we denote by $\s_{l,n-l}$ the longest element of $\SG_{l,n-l}$. 

\bigskip

\begin{prop}\label{mult w0}
Let $w \in W_n$ and let $l=\ell_t(w)$. Then:
\begin{itemize}
\itemth{a} $\ell_t(w_nw)=n-l$. 

\itemth{b} $\a_{w_n w} = \a_w \s_n \s_{n-l,l}$ and 
$\b_{w_nw}=\b_w \s_n \s_{n-l,l}$.

\itemth{c} $\s_{w_n w} = \s_n \s \s_n^{-1} \s_{n-l,l}$.

\itemth{d} Let $1 \le i_1 < \cdots < i_l \le n$ be the sequence
such that $\a_w a_l = r_{i_1} \cdots r_{i_l}$. Then 
$\a_{w_nw}=r_{j_1} \cdots r_{j_{n-l}}$, where 
$1 \le j_1 < \cdots < j_{n-l} \le n$ is the sequence such that 
$\{i_1,\dots,i_l\} \cup \{j_1,\dots,j_{n-l}\} = [1,n]$.
\end{itemize}
\end{prop}

\bigskip

\begin{proof}
(a) is clear. (d) follows from Lemma \ref{decomposition}. 
We now prove (b) and (c) simultaneously. 
For this, let $\a'=\a_w \s_n \s_{n-l,l}$, $\b'=\b_w \s_n \s_{n-l,l}$ 
and $\s'=\s_n \s_w \s_n^{-1} \s_{n-l,l}$. By the unicity statement 
of the Lemma \ref{decomposition}, we only need to show the following 
three properties:
\begin{quotation}
\begin{itemize}
\itemth{1} $\a'$, $\b' \in Y_{n-l,l}$.

\itemth{2} $\s' \in \SG_{n-l,l}$.

\itemth{3} $w_n w= \a' a_{n-l} \s' \b^{\prime -1}$.
\end{itemize}
\end{quotation}
For this, note first 
$$\s_n \SG_{l,n-l} \s_n^{-1} = \SG_{n-l,l},$$
so that (2) follows immediately. This also implies that 
$\s_n \s_{n-l,l} \s_n^{-1}=\s_{l,n-l}$ because conjugacy by 
$\s_n$ in $\SG_n$ preserves the length. 

\medskip

Let us now show (1). Let $i \in \{1,2,\dots,n\}\setminus\{n-l\}$. 
We want to show that $\ell(\a' s_i) > \ell(\a')$. By Lemma \ref{longueur}, 
this amounts to show that $\a'(i+1) > \a'(i)$. But 
$\a'=\a_w \s_{l,n-l} \s_n$. Also 
$\s_n(i) =n+1-i > \s_n(i+1)=n-i$ and $n+1-i$ and $n-i$ both belong to 
the same interval $[1,l]$ or $[l+1,n]$. 
Hence $\s_{l,n-l}\s_n(i) < \s_{l,n-l}\s_n(i+1)$ and 
$\a_w\s_{l,n-l}\s_n(i) < \a_w\s_{l,n-l}\s_n(i+1)$ since $\a_w \in Y_{l,n-l}$. 
This shows that $\a' \in Y_{n-l,l}$. Similarly, $\b' \in Y_{n-l,l}$. 
So (1) is proved.

\medskip

It remains to show (3). We have
\eqna
\a' a_{n-l} \s' \b^{\prime -1} &=& 
(\a_w \s_n \s_{n-l,l}) \cdot a_{n-l} \cdot (\s_n \s_w \s_n^{-1} \s_{n-l,l}) 
\cdot (\s_{n-l,l}^{-1} \s_n^{-1} \b_w^{-1}) \\
&=& \a_w \s_n \s_{n-l,l} a_{n-l} \s_n \s_w \b_w^{-1}
\endeqna
But $\s_{n-l,l}=\s_{[n-l+1,n]} \s_{n-l}$ and 
$\s_n \s_{[n-l+1,n]} \s_n^{-1} = \s_{[1,l]}=\s_l$ so 
\eqna
\a' a_{n-l} \s' \b^{\prime -1} &=& \a_w \s_l \s_n \s_{n-l}a_{n-l}\s_n^{-1} 
\s_w \b_w^{-1} \\
&=& \a_w \s_l \s_n w_{n-l}\s_n^{-1} \s_w \b_w^{-1},
\endeqna
the last equality following from \ref{wn an sn}. Now, 
$\s_n w_{n-l}\s_n^{-1} = w_l w_n$ (see again \ref{wn an sn}) so 
\eqna
\a' a_{n-l} \s' \b^{\prime -1} &=& \a_w \s_l w_l w_n\s_w \b_w^{-1}\\
&=& \a_w a_l w_n \s_w \b_w^{-1}\\
&=& w_n \a_w a_l \s_w \b_w^{-1}= w_n w,
\endeqna
the second equality following from \ref{wn an sn} and the third one 
from the fact that $w_n$ is central (see \ref{centre}).
\end{proof}

\bigskip

\soussection{Subgroups ${\boldsymbol{W_m}}$ of ${\boldsymbol{W_n}}$} 
If $m \le n$, we shall view $W_m$ naturally as a subgroup of $W_n$ 
(the pointwise stabilizer of $[m+1,n]$). It is the 
standard parabolic subgroup generated 
by $S_m=\{t,s_1,\dots,s_{m-1}\}$: we denote by $X_n^{(m)}$ the 
set of $w \in W_n$ which are of minimal length in $wW_m$. 
For simplification, we set $X_n=X_n^{(n-1)}$. It follows from 
Lemma \ref{longueur} that:

\bigskip

\begin{lem}\label{caracterisation X}
Let $w$ be an element of $W_n$. Then $w$ belongs to $X_n^{(m)}$ 
if and only if $0 < w(1) < w(2) < \cdots < w(m)$.
\end{lem}

\bigskip

If $I=\{i_1,\dots,i_l\} \subseteq [1,n-1]$ with $i_1 < \cdots < i_l$, 
then we set
$$c_I=s_{i_1}s_{i_2}\cdots s_{i_l}\quad\text{and}\quad 
d_I=s_{i_l} \cdots s_{i_2}s_{i_1}.$$
By convention, $c_\vide=d_\vide=1$. We have
\equat\label{xn}
X_n=\{c_{[i,n-1]}~|~1 \le i \le n\} \hskip1mm\dot{\cup}\hskip1mm
\{d_{[1,i]}tc_{[1,n-1]}~|~0 \le i \le n-1\}.
\endequat

\bigskip

\soussection{Hecke algebra} 
We fix a totally ordered abelian group $\G$ (denoted additively) 
and a weight function 
$\ph : W_n \to \G$. We set 
$$\ph(t)=b\quad\text{and}\quad 
\ph(s_1)=a \quad(=\ph(s_2)=\cdots=\ph(s_{n-1})).$$
Note that
\equat\label{phi w}
\ph(w)=\ell_t(w) b + \ell_s(w) a
\endequat
for all $w \in W_n$. 

We denote by $A$ the group algebra $\ZM[\G]$. We shall use 
an exponential notation: $A=\DS{\mathop{\oplus}_{\g \in \G}} \ZM e^\g$, 
where $e^\g\cdot e^{\g'}=e^{\g+\g'}$ for all $\g$, $\g' \in \G$. 
We set 
$$Q=e^b\quad\text{and}\quad q=e^a.$$
Note that $Q$ and $q$ are not necessarily algebraically independent.
We set
$$A_{<0}=\mathop{\oplus}_{\g < 0} \ZM e^\g,$$
and we define similarly $A_{\leqslant 0}$, $A_{> 0}$ and 
$A_{\geqslant 0}$. 

\medskip

We shall denote by $\HC_n$ the {\it Hecke algebra of $W_n$ with 
parameter $\ph$}: it is the free $A$-module with basis $(T_w)_{w \in W_n}$ 
and the multiplication is $A$-bilinear and is completely 
determined by the following rules:
$$\begin{cases}
T_w T_{w'} = T_{ww'} & \text{if $\ell(ww')=\ell(w)+\ell(w')$,}\\
(T_t - Q)(T_t+Q^{-1})=0, & \\
(T_{s_i}-q)(T_{s_i}+q^{-1}) = 0 & \text{if $1 \le i \le n-1$.}
\end{cases}$$
We also set
$$\HC_n^{<0}=\mathop{\oplus}_{w \in W_n} A_{<0} T_w.$$
Finally, we denote by $\overline{\vphantom{A}~} : \HC_n \to \HC_n$ 
the unique $A$-semilinear involution of $\HC_n$ such that 
$\overline{e^\g}=e^{-\g}$ and $\overline{T}_w=T_{w^{-1}}^{-1}$ for 
all $\g \in \G$ and $w \in W_n$. 

\bigskip

\soussection{Kazhdan-Lusztig basis} 
We shall recall here the basic definitions of Kazhdan-Lusztig theory. 
If $w \in W_n$, then \cite[Theorem 5.2]{lusztig} 
there exists a unique $C_w \in \HC_n$ such that
$$\begin{cases}
\overline{C}_w=C_w & \\
C_w \equiv T_w \mod \HC_n^{<0}.
\end{cases}$$
Note that \cite[\S 5.3]{lusztig}
\equat\label{bruhat}
C_w - T_w \in \mathop{\oplus}_{w' < w} A_{<0} T_{w'},
\endequat
where $\le$ denotes the Bruhat order on $W_n$. 
In particular, $(C_w)_{w \in W_n}$ is an $A$-basis of $\HC_n$, 
called the {\it Kazhdan-Lusztig basis} of $\HC_n$. 

\bigskip

\soussection{Cells} 
If $x$, $y \in W_n$, then we shall write 
$x \gauche y$ (resp. $x \droite y$, resp. $x \bilatere y$) if there exists 
$h \in \HC_n$ such that the coefficient of $C_x$ in the decomposition 
of $hC_y$ (resp. $C_y h$, resp. $hC_y$ or $C_y h$) is non-zero. 
We denote by $\pre{L}$ 
(resp. $\pre{R}$, resp. $\pre{LR}$) the transitive closure of $\gauche$ 
(resp. $\droite$, resp. $\bilatere$). 
Then $\pre{L}$, $\pre{R}$ and $\pre{LR}$ are preorders on $W_n$ 
and we denote respectively by $\sim_L$, $\sim_R$ and $\sim_{LR}$ the associated 
equivalence relations \cite[Chapter 8]{lusztig}. An equivalence class 
for $\sim_L$ (resp. $\sim_R$, resp. $\sim_{LR}$) is called 
a {\it left} (resp. {\it right}, resp. {\it two-sided}) {\it cell}. 
We recall the following result \cite[\S 8.1]{lusztig}: 
if $x$, $y \in W_n$, then
\equat\label{siml simr}
x \sim_L y \Longleftrightarrow x^{-1} \sim_R y^{-1}.
\endequat

\bigskip

\soussection{Domino insertion} 
If $r \ge 0$ and $w \in W_n$, then the domino insertion algorithm 
(see \cite{lam}, \cite{vanl}, \cite{white}) 
into the $2$-core $\d_r=(r,r-1,\dots,2,1)$ associates 
to $w$ a standard domino tableau $D_r(w)$ (with $n$ dominoes, 
filled with $\{1,2,\dots,n\}$). If $D$ is a domino tableau, we denote 
by $\shape(D)$ its {\it shape}: we shall denote by $\shape_r(w)$ 
the shape of $D_r(w)$ (which is equal to the shape of $D_r(w^{-1})$, 
loc. cit.). 

If $x$ and $y \in W_n$ we shall write $x \sim_L^r y$ 
(resp. $x \sim_R^r y$, resp. $x \sim_{LR}^r y$) 
if $D_r(x^{-1})=D_r(y^{-1})$ 
(resp. $D_r(x)=D_r(y)$, resp. $\shape_r(x)=\shape_r(y)$). 
These are equivalence relations on $W_n$. Note that 
$\sim_{LR}^r$ is the equivalence relation generated by 
$\sim_L^r$ and $\sim_R^r$. 

We denote by 
$\approx_L^{r+1}$ (resp. $\approx_R^{r+1}$, resp. $\approx_{LR}^{r+1}$) 
the equivalence relation generated by $\sim_L^r$ and $\sim_L^{r+1}$ 
(resp. $\sim_R^r$ and $\sim_R^{r+1}$, resp. 
$\sim_{LR}^r$ and $\sim_{LR}^{r+1}$). 
Recall the following conjecture from \cite[Conjectures A and B]{bgil}:

\bigskip

\begin{quotation}
\noindent{\bf Conjecture.} 
{\it Assume that $a$, $b > 0$. Let $r \ge 0$ and $? \in \{L,R,LR\}$.
\begin{itemize}
\itemth{a} If $ra < b < (r+1)a$, then the relations 
$\sim_?$ and $\sim_?^r$ coincide.

\itemth{b} If $r\ge 1$ and $b=ra$, then the relations 
$\sim_?$ and $\approx_?^r$ coincide.
\end{itemize}}
\end{quotation}

\bigskip

The main result of this paper is the following 
partial result towards the previous conjecture:

\bigskip

\begin{theo}\label{main}
Assume that $a$, $b > 0$. Let $r \ge 0$, $? \in \{L,R,LR\}$ and 
$x$, $y \in W_n$. Then:
\begin{itemize}
\itemth{a} If $ra < b < (r+1)a$ and $x \sim_?^r y$, then 
$x \sim_? y$.

\itemth{b} If $r \ge 1$, $b=ra$ and $x \approx_?^r y$, then 
$x \sim_? y$.
\end{itemize}
\end{theo}

\bigskip

The other sections of this paper are devoted to the proof of this 
theorem. 

\bigskip

\noindent{\sc Comments - } If one assumes Lusztig's Conjectures 
{\bf P1}, {\bf P2},\dots, {\bf P15} in \cite[Chapter 14]{lusztig}, 
then Theorem \ref{main} implies that the statement (a) of the Conjecture 
is true. Indeed, Lusztig's Conjectures imply in this case that 
the left cell representations are irreducible, and one can conclude 
by a counting argument. It might be probable that a similar 
argument applies for the statement (b), using results of 
Pietraho \cite{pietraho 2}: however, we are not able to do it.

\medskip

In the case where $b > (n-1) a$, Theorem \ref{main} was 
proved in \cite[Theorem 7.7]{lacriced} (in fact, the conjecture 
was also proved) by using a counting argument. The proof here 
does not make use of the counting argument.\finl

\bigskip

\section{Kazhdan-Lusztig polynomials, structure constants}

\medskip

\begin{quotation}
\noindent{\bf Hypothesis and notation.} 
{\it From now on, and until the end of this paper, we assume that 
$a$, $b$ are positive. Recall that $Q=e^b$ and $q=e^a$, so that 
$\ZM[Q,Q^{-1},q,q^{-1}] \subseteq A$. If $p \in A_{\ge 0}$, we 
denote by $\t_A(p)$ the coefficient of $1=e^0$ in the expansion 
of $p$ in the basis $(e^\g)_{\g \in \G}$.}
\end{quotation}

\bigskip

\soussection{Recollection of general facts} 
If $x$ and $y$ are elements of $W_n$, we set
$$C_x C_y=\sum_{z \in W_n} h_{x,y,z} C_z,$$
where the $h_{x,y,z}$ belong to $A$ and satisfy
$$\overline{h_{x,y,z}}=h_{x,y,z}.$$
We also set 
$$C_y=\sum_{x \in W_n} p_{x,y}^* T_x\quad
\text{and}\quad p_{x,y}=e^{\ph(y)-\ph(x)} p_{x,y}^*.$$
Recall \cite[Proposition 5.4]{lusztig} that 
\equat\label{kl pol}
\begin{cases}
p_{y,y}^*=p_{y,y}=1, &\\
p_{x,y}^* \in A_{< 0} & \text{if $x \neq y$.}\\
p_{x,y}^*= p_{x,y} =0 & \text{if $x \not\leqslant y$,}\\
p_{x,y} \in  A_{\geqslant 0},&\\
\t_A(p_{x,y}) = 1 & \text{if $x \le y$.}\\
\end{cases}
\endequat
Now, if $s \in S_n$, Lusztig \cite[Proposition 6.3]{lusztig} 
has defined inductively a family of polynomials 
$(M_{x,y}^s)_{sx < x < y < sy}$ by the following properties:
\refstepcounter{theo}$$\overline{M_{x,y}^s} = M_{x,y}^s,
\leqno{\boldsymbol{(\arabic{section}.\arabic{theo}\ab)}}\label{a}$$
$$M_{x,y}^s + 
\sum_{\substack{x < z < y\\ sz < z}} p_{x,z}^* M_{z,y}^s - 
e^{\ph(s)} p_{x,y}^* \in A_{<0}.
\leqno{\boldsymbol{(\arabic{section}.\arabic{theo}\bb)}}\label{b}$$
With this notation, we have \cite[Theorem 6.6]{lusztig}:

\bigskip

\begin{theo}[Kazhdan-Lusztig, Lusztig]\label{kl plus}
Let $s \in S_n$ and let $y \in W_n$. Then:
\smallskip
\begin{itemize}
\itemth{a} $C_s C_y = \begin{cases}
C_{sy} + \DS{\sum_{sx < x < y}} M_{x,y}^s C_x & \text{if $sy > y$,}\\
(e^{\ph(s)}+e^{-\ph(s)})~C_y & \textit{if $sy < y$.}\\
\end{cases}$
\smallskip
\itemth{b} If $sy < y$, and if $x \le y$, then 
$$p_{x,y}=\begin{cases}
q^2 p_{x,sy} + p_{sx,sy} - 
\DS{\sum_{\substack{x \leqslant z < sy\\ sz < z}}} 
e^{\ph(y)-\ph(z)} p_{x,z} M_{z,sy}^s & \text{if $sx < x$,}\\
p_{sx,y} & \text{if $sx > x$},
\end{cases}$$
$$p_{x,y}^*=\begin{cases}
q p_{x,sy}^* + p_{sx,sy}^* - 
\DS{\sum_{\substack{x \leqslant z < sy\\ sz < z}}} 
p_{x,z}^* M_{z,sy}^s & \text{if $sx < x$,}\\
e^{-\ph(s)} p_{sx,y}^* & \text{if $sx > x$.}
\end{cases}\leqno{\text{\it and}}$$
\end{itemize}
\end{theo}

\bigskip

\begin{coro}\label{kl}
If $s$, $s' \in \{s_1,\dots, s_{n-1}\}$ and $x$, $y \in W_n$ 
are such that $sx < x < s'x=y < sy$, then $x \sim_L y$.
\end{coro}

\bigskip

\begin{proof}
See \cite[Proposition 5 (b)]{lusztig left}.
\end{proof}

\bigskip

\soussection{Special features in type ${\boldsymbol{B}}$} 
The previous results of this section hold for any Coxeter group 
(finite or not). In this subsection, we shall investigate 
what is implied by the structure of $W_n$. The particular 
ingredient we shall need is the following lemma \cite[\S 4]{lacriced}:

\bigskip

\begin{lem}\label{al min}
$\{a_l~|~0 \le l \le n\}$ is the set of elements $w\in W_n$ which 
have minimal length in $\SG_n w \SG_n$. If $x < a_l$ for some 
$l \in \{1,2,\dots,n\}$ and some $x \in W_n$, then 
$\ell_t(x) < \ell_t(a_l)=l$. 
\end{lem}

\bigskip

It has the following consequences (here, if $p \in \ZM[q]$, 
we denote by $\deg_q p$ its degree in the variable $q$):

\bigskip

\begin{coro}\label{pxy}
Let $x$ and $y$ be two elements of $W_n$ such that 
$x \le y$ and $\ell_t(x)=\ell_t(y)$. Then:
\begin{itemize}
\itemth{a} $p_{x,y} \in \ZM[q]$ and, if $x \neq y$, then 
$\deg_q p_{x,y} < \ell(y)-\ell(x)$.

\itemth{b} If $1 \le i \le n-1$ is such that 
$s_ix < x < y < s_iy$, then $M_{x,y}^{s_i} \in \ZM$: it is 
equal to $\t_A(q p_{x,y}^*)$ (note also that $qp_{x,y}^* \in \ZM[q^{-1}]$).
\end{itemize}
\end{coro}

\bigskip

\begin{proof}
We shall prove (a) and (b) together by induction on the pair 
$(\ell(y),\ell(y)-\ell(x))$ (with lexicographic order). 
The result is obvious if $\ell(y)=\ell(x)$ or if $\ell(y) \le 1$. 
So assume now that $\ell(y) > 1$, that $\ell(y)-\ell(x) > 0$ and 
that (a) and (b) hold for all pairs $(x',y')$ such that 
$(\ell(y'),\ell(y')-\ell(x')) < (\ell(y),\ell(y)-\ell(x))$. 
First, note that
$$e^{\ph(y)-\ph(x)}=q^{\ell(y)-\ell(x)},$$
because $\ph(y)-\ph(x)=(\ell_t(y)-\ell_t(x)) b + (\ell_s(y)-\ell_s(x)) a
= (\ell_s(y)-\ell_s(x)) a=(\ell(y)-\ell(x)) a$. 

Let us first prove (a). So we have $x < y$ and $\ell_t(x)=\ell_t(y)$. By 
Lemma \ref{al min}, this implies that there exists 
$i \in \{1,2,\dots,n-1\}$ such that $s_i y < y$ or $ys_i < y$. 
In the second case, one can exchange $y$ and $y^{-1}$ 
(and $x$ and $x^{-1}$) by using \cite[\S 5.6]{lusztig}, 
so that we may assume that $s_i y < y$. Then, Theorem \ref{kl plus} (b) 
can be rewritten as follows:
$$p_{x,y}=\begin{cases}
(q^2 p_{x,s_iy} - q^{\ell(y)-\ell(x)} M_{z,s_iy}^{s_i}) +
 p_{s_ix,s_iy} - \DS{\sum_{\substack{x < z < s_iy\\ s_iz < z}}} 
q^{\ell(y)-\ell(z)} p_{x,z} M_{z,s_iy}^{s_i} & \text{if $s_ix < x$,}\\
p_{s_ix,y} & \text{if $s_ix > x$}.
\end{cases}$$
If $s_i x > x$, then the result follows from 
the induction hypothesis. If $s_i x < x$, then 
$$q^2 p_{x,s_iy} - q^{\ell(y)-\ell(x)} M_{x,s_iy}^{s_i} = 
q^{\ell(y)-\ell(x)}(qp_{x,s_iy}^* - M_{x,s_iy}^{s_i})$$
belong to $\ZM[q]$ and 
has degree $< \ell(y)-\ell(x)$ by the induction hypothesis. 
The other terms in the above formula also belong to $\ZM[q]$ 
and also have degree $< \ell(y)-\ell(x)$ by the induction hypothesis. 
So we get (a). 

Let us now prove (b). So we assume that $s_i x < x < y < s_i y$. 
Then, using the induction hypothesis and \ref{kl pol}, 
the condition \ref{b} (b) can be rewritten
$$M_{x,y}^{s_i}- q p_{x,y}^* \in A_{<0}.$$
Now, the result follows easily from (a).
\end{proof}

\bigskip

Now, if $tx < x < y < ty$ 
are such that $\ell_t(x)=\ell_t(y)$, let us define 
an element $\mu_{x,y} \in A$ by induction on $\ell(y)-\ell(x)$ 
by the following formula:
$$\mu_{x,y} = p_{x,y}-\sum_{\substack{x < z < y\\ tz < z}} 
p_{x,z} \mu_{z,y}.$$
It follows easily from Corollary \ref{pxy} (and an induction 
argument on $\ell(y)-\ell(x)$) that
\equat\label{mu q}
\mu_{x,y} \in \ZM[q]\quad\text{and}\quad
\deg_q \mu_{x,y} < \ell(y)-\ell(x).
\endequat
Moreover:

\bigskip

\begin{coro}\label{mt}
Assume that $tx < x < y < ty$ and that $\ell_t(x)=\ell_t(y)$. 
Then:
\begin{itemize}
\itemth{a} If $b > (\ell(y)-\ell(x)) a$, then 
$M_{x,y}^t = Q q^{\ell(x)-\ell(y)} \mu_{x,y} + 
Q^{-1} q^{\ell(y)-\ell(x)} \overline{\mu_{x,y}}$.

\itemth{b} If $b = (\ell(y)-\ell(x)) a$, then 
$M_{x,y}^t = \mu_{x,y} + \overline{\mu_{x,y}} - \t_A(\mu_{x,y})$.
\end{itemize}
\end{coro}

\bigskip

\begin{proof}
Let us assume that $b \ge (\ell(y)-\ell(x)) a$. We shall 
prove the result by induction on $\ell(y)-\ell(x)$. 
By the induction hypothesis, the condition \ref{b} (b) can 
we written
$$M_{x,y}^t - Q q^{\ell(x)-\ell(y)} p_{x,y} + 
\sum_{\substack{x < z < y\\ tz < z}} 
p_{x,z}^* \Bigl(Q q^{\ell(z)-\ell(y)} \mu_{z,y} + 
Q^{-1} q^{\ell(y)-\ell(z)} \overline{\mu_{z,y}}\Bigr) \in A_{<0}.$$
But, if $x < z < y$ and $tz < z$, then 
$$p_{x,z}^* Q^{-1} q^{\ell(y)-\ell(z)} \overline{\mu_{z,y}}
\in A_{<0}$$
because $p_{x,z}^* \in A_{<0}$, $\overline{\mu_{z,y}} \in A_{\leqslant 0}$ 
and $Q^{-1} q^{\ell(y)-\ell(z)}=e^{-b+(\ell(y)-\ell(z))a} \in A_{< 0}$ 
(since $\ell(y)-\ell(z) < \ell(y)-\ell(x)$). 
Therefore, 
$$M_{x,y}^t - Q q^{\ell(x)-\ell(y)} p_{x,y} + 
\sum_{\substack{x < z < y\\ tz < z}} 
Q q^{\ell(x)-\ell(y)} p_{x,z} \mu_{z,y}  \in A_{<0}.$$
In other words, 
$$M_{x,y}^t - Q q^{\ell(x)-\ell(y)} \mu_{x,y} \in A_{< 0}.$$
Let $\mu=Q q^{\ell(x)-\ell(y)} \mu_{x,y}$. 
Two cases may occur:
\begin{itemize}
\item[$\bullet$] If $b > (\ell(y)-\ell(x))a$, then $\mu \in A_{>0}$ 
and so the condition \ref{a} (a) forces 
$M_{x,y}^t=\mu + \overline{\mu}$, as required.

\item[$\bullet$] If $b = (\ell(y)-\ell(x))a$, then 
$\mu=\mu_{x,y} \in A_{\geqslant 0}$ and now the condition \ref{a} (a) forces 
$M_{x,y}^t=\mu + \overline{\mu}-\t_A(\mu)$, as required.
\end{itemize}
The proof of the Corollary is complete. 
\end{proof}

\bigskip

We conclude this subsection with two results involving 
the decomposition of Lemma \ref{decomposition}. 

\bigskip

\begin{lem}\label{remove b}
Let $x$ and $y$ be two elements of $W_n$ and let $s \in S_n$ 
be such that $sx<x<y<sy$, $\ell_t(x)=\ell_t(y)$ and 
$\b_x=\b_y=\b$. Then $M_{x,y}^s=M_{x\b,y\b}^s$ (note that 
$\b_{x\b}=\b_{y\b}=1$).
\end{lem}

\bigskip

\begin{proof}
See \cite[Proposition 7.2]{lacriced}. Strictly speaking, 
in \cite{lacriced}, the authors are generally working 
with a special choice of a function $\ph$ (``asymptotic case''): 
however, the 
reader can check that the proof of this particular result, namely 
\cite[Proposition 7.2]{lacriced}, remains valid for all 
choices of parameters. 
\end{proof}

\bigskip

\begin{prop}\label{aalsb}
Let $l \in [0,n]$, let $\s$ and $\s' \in \SG_{l,n-l}$ be such 
that $\s \sim_L \s'$ and let $\b \in Y_{l,n-l}$. Then
$$a_l \s \b^{-1} \sim_L a_l \s' \b^{-1}.$$
\end{prop}

\bigskip

\begin{proof}
By the description of Kazhdan-Lusztig left cells in the symmetric 
group \cite[Theorem 1.4 and \S 4]{KL}, we may assume that there exist 
two elements 
$s$ and $s'$ in $\{s_1,\dots,s_{l-1},s_{l+1},\dots,s_{n-1}\}$ such that 
$\s'=s'\s$ and $s\s < \s < \s' < s\s'$. Let $u=a_lsa_l$ and 
$u'=a_ls'a_l$. Then $u$ and $u'$ belong to 
$\{s_1,\dots,s_{l-1},s_{l+1},\dots,s_{n-1}\}$ by \ref{al s}, and
$$u a_l \s \b^{-1} < a_l \s \b^{-1} < u'a_l \s \b^{-1}=a_l \s' \b^{-1}
< u a_l \s' \b^{-1}.$$
So $(*)$ follows from Corollary \ref{kl}.
\end{proof}

\bigskip

\soussection{${\boldsymbol{*}}$-operation\label{star sec}} 
We shall recall the definition of the $*$-operation (see \cite[\S 4]{KL}) 
and prove some properties which are particular to the type $B$. 
Let us introduce 
some notation. If $1 \le i \le n-2$ and $x \in W_n$, we set 
$$\RC_i(x)=\{s \in \{s_i,s_{i+1}\}~|~\ell(xs) < \ell(x)\}.$$
We denote by $\DC_i(W_n)$ the set of $x \in W_n$ such that 
$|\RC_i(x)|=1$. If $x \in \DC_i(W_n)$, then it is readily 
seen that the set $\{xs_i,xs_{i+1}\} \cap \DC_i(W_n)$ is 
a singleton. We shall denote by $\g_i(x)$ the unique element 
of this set (it is denoted by $x^*$ in \cite[\S 4]{KL}, but 
we want to emphasize that it depends on $i$). Note that 
$$\g_i \circ \g_i = \Id_{\DC_i(W_n)}.$$
We recall Kazhdan-Lusztig result \cite[Corollary 4.3]{KL}: if 
$x$ and $y \in \DC_i(W_n)$, then
\equat\label{star}
x \sim_L y \Longleftrightarrow \g_i(x) \sim_L \g_i(y).
\endequat
The fact that $t$ is not conjugate to any of the $s_k$'s implies 
the following easy fact:

\bigskip

\begin{prop}\label{descente t}
Let $x \in W_n$ and let $1 \le k \le n-1$. Then 
$xs_k > x$ if and only if $txs_k > tx$.
\end{prop}

\bigskip

\begin{proof}
Indeed, by Lemma \ref{longueur}, we have $xs_k > x$ if and only if 
$x(k) < x(k+1)$. But, for any $j \in I_n^+$, there is no element 
$j' \in I_n$ such that $t(j) < j' < j$. So $x(k) < x(k+1)$ 
if and only if $tx(k) < tx(k+1)$ that is, if and only if 
$txs_k > tx$ (again by Lemma \ref{longueur}).
\end{proof}

\bigskip

The proposition \ref{descente t} implies immediately the 
following result:

\bigskip

\begin{coro}\label{star t}
Let $x \in W_n$ and let $1 \le i \le n-2$. 
Then $x \in \DC_i(W_n)$ if and only if $tx \in \DC_i(W_n)$. 
If this is the case, then $\g_i(tx)=t\g_i(x)$.
\end{coro}

\bigskip

\soussection{Two relations ${\boldsymbol{\gauche}}$} 
The crucial steps towards the proof of Theorem \ref{main} are 
the following two propositions, whose proofs will be given in sections 
\ref{M1} and \ref{M2} respectively.

\bigskip

\begin{prop}\label{exemples m1}
Let $l \in \{1,\dots,n-1\}$ and assume that $b \ge (n-1) a$. 
Then $M_{r_1\cdots r_l \s_{[l+1,n]},r_2 \dots r_l r_n \s_{[l+1,n]}}^t \neq 0$. 
\end{prop}

\bigskip

\begin{prop}\label{exemples m2}
Let $l \in \{1,\dots,n-1\}$ and assume that $(n-2) a < b \le (n-1) a$. 
Then $a_{l-1}\s_{[l,n]} \gauche a_l \s_{[l,n]}$.
\end{prop}

\bigskip

\section{Proof of Proposition \ref{exemples m1}\label{M1}}

\medskip

\begin{quotation}
\noindent{\bf Notation.} 
{\it If $u$, $v \in W_n$ are such that $u \le v$, we denote by 
$[u;v]$ the Bruhat interval between $u$ and $v$. In this section,
and only in this section, we assume that $l \ge 1$ and $b \ge (n-1)a$ 
and we set $x=r_1\cdots r_l \s_{[l+1,n]}$ 
and $y=r_2 \dots r_l r_n \s_{[l+1,n]}$.}
\end{quotation}

\bigskip

\soussection{Easy reduction} 
Note that  
$$tx < x < y < ty,$$
so it makes sense to compute $M_{x,y}^t$. Moreover, 
$\ell(y)-\ell(x) = n-1$ so, by Corollary \ref{mt}, 
we only need to prove that $\m_{x,y} \neq 0$ 
(even if $b=(n-1)a$). For this, we only need to 
show that
$$\t_A(\mu_{x,y}) \neq 0.\leqno{(?)}$$

\bigskip

\soussection{The Bruhat interval ${\boldsymbol{[x;y]}}$} 
First, note that 
$$x=a_l\s_{[l+1,n]}=\s_{[l+1,n]}a_l$$
$$y=s_1\cdots s_{l-1} s_{n-1}\cdots s_l x= 
s_1\cdots s_{l-1} s_{n-1}\cdots s_l \s_{[l+1,n]}a_l = 
c_{[1,l-1]} \s_{[l,n]} a_l.\leqno{\text{and}}$$
Since $a_l$ has minimal length in $\SG_n a_l$, the map
$$\fonctio{[\s_{[l+1,n]};c_{[1,l-1]} \s_{[l,n]}]}{[x;y]}{z}{za_l}$$
is an increasing bijection \cite[Lemma 9.10 (f)]{lusztig}. 
Since the support of $c_{[1,l-1]}$ is disjoint from 
the support of $\s_{[l,n]}$, the map
$$\fonctio{[1;c_{[1,l]}] 
\times [\s_{[l+1,n]};\s_{[l,n]}]}{[\s_{[l+1,n]};c_{[1,l-1]}
\s_{[l,n]}]}{(z,z')}{zz'}$$
is an increasing bijection (for the product order). Now, 
$\s_{[l,n]}$ is the longest element of $\SG_{[l,n]}$ and 
$\s_{[l+1,n]}\s_{[l,n]}=c_{[l,n-1]}$. Therefore, the map 
$$\fonctio{[1;c_{[l,n-1]}]}{[\s_{[l+1,n]};\s_{[l,n]}]}{z}{z\s_{[l,n]}}$$
is a decreasing bijection. So, if we denote by $\PC(E)$ 
the set of subsets of a set $E$, then the maps
$$\fonctio{\PC([1,l-1])}{[1;c_{[1,l-1]}]}{I}{c_I}$$
$$\fonctio{\PC([l,n-1])}{[\s_{[l+1,n]};
\s_{[l,n]}]}{J}{c_{\bar{J}}\s_{[l,n]}}
\leqno{\text{and}}$$
are increasing bijections (here, $\bar{J}$ denotes the complement of $J$). 
On the other hand, the map 
$$\fonctio{\PC([1,l-1])\times\PC([l,n-1])}{\PC([1,n-1])}{(I,J)}{I \cup J}$$
is an increasing bijection. Finally, by composing all these bijections, 
we get an isomorphism of ordered sets
$$\fonction{\a}{\PC([1,n-1])}{[x;y]}{I}{c_{I \cap [1,l-1]} 
c_{\overline{I \cap [l,n-1]}} \s_{[l,n]} a_l.}$$

\bigskip

\soussection{The elements ${\boldsymbol{z \in [x;y]}}$ such that 
${\boldsymbol{tz<z}}$} 
If $I \subseteq [1,n-1]$ is such that $t\a(I) < \a(I)$, 
we set $\mut_I=\t_A(\mu_{\a(I),y})$. So we can rephrase $(?)$ as follows:
$$\mut_\vide \neq 0.\leqno{(??)}$$
But, by the induction formula that defines the $\mu$-polynomials 
and by \ref{kl pol}, we have, for all $I \subseteq [1,n-1]$ 
such that $t\a(I) < \a(I)$, 
\equat\label{mutilde I}
\mut_I = 1 -\sum_{\substack{I \varsubsetneq J \subseteq [1,n-1]\\ t\a(J) 
< \a(J)}} \mut_J.
\endequat
Let 
$$\EC=\{I \in \PC([1,n-1])~|~t\a(I) < \a(I)\}.$$
The set $\EC$ is easy to describe:

\bigskip

\begin{lem}\label{alpha I}
Let $I \subseteq [1,n-1]$. Then $t \a(I) > \a(I)$ if and only 
if $[1,l-1] \varsubsetneq I$. 
\end{lem}

\bigskip

\begin{proof}[Proof of Lemma \ref{alpha I}]
By Lemma \ref{longueur}, we just need to show that 
$$\text{\it $\a(I)^{-1}(1) > 0$ if and only if $[1,l-1] \varsubsetneq I$.}
\leqno{(\#)}$$
For simplification, we set 
$A=[1,l-1] \cap I$ and $B=\overline{I \cap [l,n-1]}$. 
So $\a(I)=c_A c_B \s_{[l,n]} a_l$. 

First, assume that $[1,l-1] \not\subseteq I$. Then 
$0 < c_A^{-1}(1) < n$, so $\s_{[l,n]}^{-1}c_B^{-1}c_A^{-1}(1)=c_A^{-1}(1)$ 
and $\a(I)^{-1}(1)=a_l^{-1}(c_A^{-1}(1)) < 0$ by \ref{al}. 
This shows $(\#)$ in this case.

Now, let us assume that $[1,l-1] = I$. Then 
$c_A=s_1\cdots s_{l-1}$ and $c_B=s_l \cdots s_{n-1}$ 
and so $c_A^{-1}(1)=l$ and $c_B^{-1}(l)=n$. In particular, 
$\a(I)^{-1}(1)=a_l^{-1}\s_{[l,n]}(n)=a_l^{-1}(l)=-1 < 0$ by \ref{al}. 
This shows $(\#)$ again in this case.

Now, let us assume that $[1,l-1] \varsubsetneq I$. 
Then $c_A^{-1}(1)=l$ and $c_B^{-1}(l) < n$ and so 
$\s_{[l,n]}^{-1}c_B^{-1}c_A^{-1}(1) > l$. So 
$\a(I)^{-1}(1) > 0$ by \ref{al}. The proof of $(\#)$ is complete.
\end{proof}

\bigskip

\soussection{Computation of the ${\boldsymbol{\mut_I}}$} 
We shall now compute the family $(\mut_I)_{I \in \EC}$ by descending induction 
on $|I|$, by using the formula \ref{mutilde I}. For this, 
the following well-known lemma will be useful.

\bigskip

\begin{lem}\label{sous-ensembles}
If $S$ is a finite set and $I \varsubsetneq S$, then 
$\DS{\sum_{I \subseteq J \subseteq S} (-1)^{|J|} = 0}$.
\end{lem}

\bigskip

To obtain the value of $\mut_\vide$, the proof 
goes in three steps.  

\bigskip

\equat\label{mutilde 1}
\text{\it If $[l,n-1] \subseteq I \varsubsetneq [1,n-1]$, 
then $\mut_I=(-1)^{n-|I|}$.}
\endequat

\bigskip

\begin{proof}[Proof of \ref{mutilde 1}]
First, note that $I \in \EC$ by Lemma \ref{alpha I}. 
We argue by descending induction on $|I|$. If $|I|=n-2$, 
then $\mut_I=1$, as desired. Now, let us assume that 
$[l,n-1] \subseteq I \varsubsetneq [1,n-1]$ and that 
$\mut_J=(-1)^{n-|J|}$ for all $I \varsubsetneq J \varsubsetneq [1,n-1]$. 
Then, by \ref{mutilde I}, we have
$$\mut_I=1-\sum_{I \varsubsetneq J \varsubsetneq [1,n-1]} (-1)^{n-|J|}.$$
Therefore,
$$\mut_I=1 + (-1)^{n-|I|} + (-1)^{n-(n-1)} - 
\sum_{I \subseteq J \subseteq [1,n-1]} (-1)^{n-|J|}=(-1)^{n-|I|},$$
the last equality following from Lemma \ref{sous-ensembles}.
\end{proof}

\bigskip

\equat\label{mutilde 2}
\text{\it If $I \in \EC$ is such that 
$[l,n-1] \not\subseteq I$ and $I \not\subseteq [1,l-1]$, 
then $\mut_I=0$.}
\endequat

\bigskip

\begin{proof}[Proof of \ref{mutilde 2}]
We shall again argue by descending induction on $|I|$. 
Let $I'=I \cup [l,n-1]$. Then, by \ref{mutilde I}, we have
$$\mut_I=1
-\sum_{\substack{J \in \EC\\ I \varsubsetneq J\text{ and }I' \subseteq J}}
\mut_J 
- \sum_{\substack{J \in \EC\\ I \varsubsetneq J\text{ and }I' \not\subseteq J}}
\mut_J.$$
But, if $J \in \EC$ is such that $I \varsubsetneq J$ and $I' \not\subseteq J$, 
(or, equivalently, $[l,n-1] \not\subsetneq J$), then $\mut_J=0$ by the 
induction hypothesis. On the other hand, 
if $J \in \EC$ is such that $I \varsubsetneq J$ and $I' \not\subseteq J$, 
then $\mut_J=(-1)^{n-|J|}$ by \ref{mutilde 1}. Therefore, 
$$\mut_I=1
-\sum_{\substack{J \in \EC\\ I \varsubsetneq J\text{ and }I' \subseteq J}}
(-1)^{n-|J|} = 1-\sum_{I' \subseteq J \varsubsetneq [1,n-1]} (-1)^{n-|J|} 
= -\sum_{I' \subseteq J \subseteq [1,n-1]} (-1)^{n-|J|}=0$$
by Lemma \ref{sous-ensembles}. 
\end{proof}

\bigskip

\equat\label{mutilde 3}
\text{\it If $I \subseteq [1,l-1]$, then $\mut_I=(-1)^{l-1-|I|}$.}
\endequat

\bigskip

\begin{proof}[Proof of \ref{mutilde 3}]
Note that $I \in \EC$. We shall argue by descending induction on $|I|$. 
First, for all $J$ such that $I \varsubsetneq J \subseteq [1,n-1]$, we have 
$t\a(I) > \a(I)$. Therefore, $\mut_{[1,l-1]}=1$, as desired.

Now, let $I \varsubsetneq [1,l-1]$ and assume that, for all 
$I \varsubsetneq J \subseteq [1,l-1]$, we have $\mut_J=(-1)^{l-1-|J|}$. 
Then
$$\mut_I=1-\sum_{\substack{J \in \EC \\ I \varsubsetneq J}} \mut_I.$$
Now, if $J \in \EC$ is such that $I \varsubsetneq J$, then three cases 
may occur:
\begin{itemize}
\item If $J \subseteq [1,l-1]$, then $\mut_J=(-1)^{l-1-|J|}$ by the induction 
hypothesis.

\item If $J \not\subseteq [1,l-1]$ and $[l,n-1] \not\subseteq I$, then 
$\mut_J=0$ by \ref{mutilde 2}.

\item If $[l,n-1] \subseteq J$, then $\mut_J=(-1)^{n-|J|}$.
\end{itemize}
Therefore, if we set $I'=I \cap [l,n-1]$, then we get
$$\mut_I=1-\sum_{I' \subseteq J \varsubsetneq [l,n-1]} (-1)^{n-|J|} 
- \sum_{I \varsubsetneq J \subsetneq [1,l-1]} (-1)^{l-1-|J|}.$$
But
$$1-\sum_{I' \subseteq J \varsubsetneq [l,n-1]} (-1)^{n-|J|}=
-\sum_{I' \subseteq J \subseteq [l,n-1]} (-1)^{n-|J|} = 0$$
$$-\sum_{I \varsubsetneq J \subsetneq [1,l-1]} (-1)^{l-1-|J|} = 
(-1)^{l-1-|I|} -\sum_{I \subseteq J \subsetneq [1,l-1]} (-1)^{l-1-|J|} 
= (-1)^{l-1-|I|}\leqno{\text{and}}$$
by Lemma \ref{sous-ensembles}. The proof is now complete.
\end{proof}

\bigskip

As a special case of \ref{mutilde 3}, we get that 
$$\mut_\vide=(-1)^{l-1}.$$
This shows (?). The proof of the Proposition \ref{exemples m1} is 
complete. 

\bigskip

\section{Consequence of Proposition \ref{exemples m1}}

\medskip

The aim of this section is to prove the following

\bigskip

\begin{prop}\label{quasi asymptotique}
Let $l \in \{0,1,\dots,n\}$, let $\a$, $\b \in Y_{l,n-l}$ and 
let $\s$ and $\s' \in \SG_{l,n-l}$ be such that $\s \sim_L \s'$. 
Assume that $b \ge (n-1) a$. Then 
$$\a a_l \s \b^{-1} \sim_L a_l \s' \b^{-1}.$$
\end{prop}

\bigskip

\noindent{\sc Remarks - } (1) The condition $\s \sim_L \s'$ 
does not depend on the choice of $a$ and $b$ in $\G$. Indeed, 
by \cite[Theorem 1]{geck induction}, $\s \sim_L \s'$ in $W_n$ 
if and only if $\s \sim_L \s'$ in $\SG_{l,n-l}$. But 
this last condition depends neither on the choice of $b$ 
(since $t \not\in \SG_{l,n-l}$) nor on the choice of $a$ 
(provided that it is in $\G_{>0})$.

\medskip

(2) If $b > (n-1) a$, then the above proposition 
is proved in \cite[Theorem 7.7]{lacriced} (see also 
\cite[Corollary 5.2]{bilatere} 
for the exact bound) by a counting argument. The proof 
below will not use this counting argument but uses instead 
the proposition \ref{exemples m1}: it allows to extend the scope 
of validity to the case where $b =(n-1) a$ (this is 
compatible with \cite[Conjecture A (b)]{semicontinu}).\finl

\bigskip

\begin{proof}
First, recall that 
$a_l \s \b^{-1} \sim_L a_l \s' \b^{-1}$ by Poposition \ref{aalsb}. 
This shows that we may (and we will) assume that $\s=\s'$. 
We want to show that $\a a_l \s\b^{-1} \sim_L a_l\s\b^{-1}$. 
We shall use induction on $n$. So let $(P_n)$ denote the 
following statement:
\begin{quotation}
\begin{itemize}
\item[$(P_n)$] For all $l \in [0,n]$, for all sequences 
$1 \le i_1 < \cdots < i_l \le n$, 
for all $\s \in \SG_{l,n-l}$ and for all $\b \in Y_{l,n-l}$, we have  
$r_{i_1}r_{i_2} \cdots r_{i_l} \s\b^{-1} \sim_L r_1r_2\cdots r_l \s\b^{-1}$.
\end{itemize}
\end{quotation}

\medskip

The property $(P_1)$ is vacuously true and the property $(P_2)$ 
can be easily checked by a straightforward computation. 
So we assume that $n \ge 3$ and $(P_m)$ holds for all $m < n$.
Now, let $l \in [0,n]$, let $1 \le i_1 < \cdots < i_l \le n$ 
be a sequence of elements of $[1,n]$, let $\s \in \SG_{l,n-l}$ 
and let $\b \in Y_{l,n-l}$. 
As a consequence of this induction hypothesis, we get:
\begin{quotation}
{\small\begin{lem}\label{induction}
If $k \in [1,l]$ is such that $i_k < n$, then 
$r_{i_1}r_{i_2} \cdots r_{i_l} \s\b^{-1} \sim_L 
r_1\cdots r_k r_{i_{k+1}} \cdots r_{i_l} \s\b^{-1}$.
\end{lem}

\begin{proof}
Let $w=r_{i_1}r_{i_2} \cdots r_{i_l} \s\b^{-1}$ and 
$w'=r_1\cdots r_k r_{i_{k+1}} \cdots r_{i_l} \s\b^{-1}$. Let 
us write $w=vx^{-1}$ and $w'=v'x^{\prime -1}$ with 
$v$, $v' \in W_{i_k}$ and $x$, $x' \in X_n^{(i_k)}$. 
First, note that 
$$w w^{\prime -1} = 
(r_{i_1} \cdots r_{i_k}) \cdot (r_1 \cdots r_k)^{-1} \in W_{i_k}.$$
Therefore, $x=x'$ and 
$$v v^{\prime -1} =
(r_{i_1} \cdots r_{i_k}) \cdot (r_1 \cdots r_k)^{-1} \in W_{i_k}.$$
Moreover, by Lemma \ref{caracterisation X}, 
we have $0 < x(1) < \cdots < x(i_k)$. So, if $i \in [1,i_k]$, then 
$v^{-1}(i) < 0$ (resp. $v^{\prime -1}(i) < 0$) if and only if 
$i \in \{i_1,\dots,i_k\}$ (resp. $\{1,\dots,k\}$). So, 
by Lemma \ref{decomposition}, we have 
$$v= r_{i_1} \cdots r_{i_k} \t \g^{-1}
\quad \text{and}\quad
v'=r_1\cdots r_k \t \g^{-1},$$
where $\t \in \SG_{k,i_k-k}$ and $\g \in Y_{k,i_k-k}$. 
But, since $i_k < n$, it follows from the induction 
hypothesis that $v \sim_L v'$. Here, note that $v \sim_L v'$ 
in $W_{i_k}$ if and only if $v \sim_L v'$ in $W_n$ 
(see \cite[Theorem 1]{geck induction}). So, by 
\cite[Proposition 9.11]{lusztig}, we get that $w \sim_L w'$.
\end{proof}

\begin{coro}\label{coro induction}
~
\begin{itemize}
\itemth{a} If $i_l < n$, then $r_{i_1}r_{i_2} \cdots r_{i_l} \s\b^{-1} \sim_L r_1r_2\cdots r_l \s\b^{-1}$.

\itemth{b} If $i_l=n$, then $r_{i_1}r_{i_2} \cdots r_{i_l} \s\b^{-1} \sim_L r_1r_2\cdots r_{l-1} r_n \s\b^{-1}$.
\end{itemize}
\end{coro}}
\end{quotation}

By Corollary \ref{coro induction}, we only need to show that 
$$r_1r_2\cdots r_{l-1} r_n \s\b^{-1} \sim_L r_1r_2\cdots r_l \s \b^{-1}.
\leqno{(?)}$$
Now, let us write $\s=(\l,\mu)$, where 
$\l \in \SG_{[1,l]}$ and $\mu \in \SG_{[l+1,n]}$.
%We shall show (?) by induction on $\ell(\l) + (\ell(\s_{[l+1,n]})-\ell(\mu))$. 
Three cases may occur:

\medskip

$\bullet$ {\it Case 1:} If $\l=1$ and $\mu=\s_{[l+1,n]}$, then $\s=\s_{[l+1,n]}$. Since 
$r_1 r_2\cdots r_{l-1}r_k \s\b^{-1}=s_k s_{k-1} \cdots s_{l+1} \s\b^{-1}$ 
for all $k > l$, we have 
\begin{multline*}
\qquad r_1 \cdots r_{l-1}r_n \s\b^{-1} \pre{L} 
r_1 \cdots r_{l-1}r_{n-1} \s\b^{-1} \pre{L} \cdots \\ 
\pre{L} r_1 \cdots r_{l-1}r_{l+1} \s\b^{-1} 
\pre{L} r_1 \cdots r_{l-1}r_l \s\b^{-1}.\qquad
\end{multline*}
On the other hand, by Proposition \ref{exemples m1} and 
Lemma \ref{remove b}, we get 
$r_1 \cdots r_{l-1}r_l\s\b^{-1}\pre{L}r_1 \cdots r_{l-1}r_n\s\b^{-1}$. 
This shows (?) in this particular case.

\medskip

$\bullet$ {\it Case 2:} 
If $\mu \neq \s_{[l+1,n]}$, then $n \ge l+2$ and 
there exists $k \in [l+1,n-1]$ such that $s_k \s > \s$. 
Let $i$ be maximal such that $s_i \s > \s$. 
We shall prove (?) by descending induction on $i$. 
For simplification, let $x=r_1\cdots r_{l-1} r_n \s \b^{-1}$. 

First, if $i = n-1$, then, by \ref{sk rl}, we have (since 
$n-2 > l-1$) 
$$s_{n-2} x = r_1\cdots r_{l-1} s_{n-2} r_n \s \b^{-1}
= r_1\cdots r_{l-1} r_n s_{n-1} \s \b^{-1} > s_{n-2} x,$$
$$s_{n-1} x = r_1\cdots r_{l-1} s_{n-1} r_n \s \b^{-1}
r_1\cdots r_{l-1} r_{n-1} \s \b^{-1} < x$$
$$s_{n-2} s_{n-1} x = r_1\cdots r_{l-1} s_{n-2} r_{n-1} \s \b^{-1} 
= r_1 \cdots r_{l-1} r_{n-2} \s \b^{-1} < s_{n-1} x.\leqno{\text{and}}$$
So $x \sim_L s_{n-1} x$ by Corollary \ref{kl}. On the other hand, by 
Corollary \ref{coro induction}, we have 
$s_{n-1} x \sim_L r_1 \cdots r_l \s \b^{-1}$, 
so we get (?) in this case.

Now, assume that $l+1 \le i < n-1$. Then $s_{i+1} \s < \s$ 
(by the maximality of $i$). Two cases may occur:

\begin{quotation}
\noindent{$\bullet$} {\it Subcase 1:} 
If $s_i s_{i+1} \s < s_{i+1} \s$, then we set $\t=s_{i+1} \s < \s$ 
and $y=r_1 \cdots r_{l-1} r_n \t \b^{-1}$. Then $y=s_i x < x$ 
by \ref{sk rl}. Moreover, still by \ref{sk rl}, we have 
$$s_{i-1} x = r_1 \cdots r_{l-1} r_n s_i \s \b^{-1} > x$$
$$s_{i-1} s_i x = r_1 \cdots r_{l-1} r_n s_i s_{i+1} \s \b^{-1} < s_i x.
\leqno{\text{and}}$$
So $x \sim_L y$ by Corollary \ref{kl}. But, by the induction 
hypothesis (and since $s_{i+1} \t > \t$), we have 
$y \sim_L a_l \t \b^{-1}$. But $\s \sim_L \t$ (again by Corollary 
\ref{kl} and since $s_i \t < \t < \s=s_{i+1} \t < s_i \s$), 
so $a_l \s \b^{-1} \sim_L a_l \t \b^{-1}$ by $(*)$. 
This shows (?).

\medskip

\noindent{$\bullet$} {\it Subcase 2:} If $s_i s_{i+1} \s > s_{i+1} \s$, 
then $s_{i+1} s_i \s > s_i \s$ (by an easy application of 
Lemma \ref{longueur}) so, if we set $\t=s_i\s$ and 
$y=r_1 \cdots r_{l-1} r_n \t \b^{-1}$, we have, by the 
induction hypothesis, $y \sim_L a_l \t\b^{-1}$. Moreover, 
$s_{i+1} \t > \t=s_i\s > \s > s_{i+1} \s$ and, 
by the same argument as in the subcase 1, we have 
$s_i y > y =s_{i-1} x > x > s_i x$. So 
$x \sim_L y$, $\s \sim_L \t$. So it follows from $(*)$ and 
$x \sim_L a_l \s \b^{-1}$, as required.
\end{quotation}

\medskip

$\bullet$ {\it Case 3:} If $\l \neq 1$, then we set 
$x=r_1 \cdots r_{l-1} r_n \s \b^{-1}$ and $y=r_1 \cdots r_l \s\b^{-1}$. 
We want to show that $x \sim_L y$. For this, let $x'=w_n x$, 
$y'=w_n y$, $\s' = \s_n \s \s_n^{-1} \s_{n-l,l}$ and 
$\b'=\b \s_n \s_{n-l,l}$. Then, by Proposition \ref{mult w0},
$$x'=r_l r_{l+1} \cdots r_{n-1} \s'\b^{\prime -1} \quad\text{and}\quad
y'=r_l r_{l+1} \cdots r_{n-2} r_n \s' \b^{\prime -1}.$$
But, by Corollary \ref{coro induction}, we have 
$$x' \sim_L r_1 \cdots r_{n-l} \s' \b^{\prime -1}
\quad\text{and}\quad
y' \sim_L r_1 \cdots r_{n-l-1} r_n \s' \b^{\prime -1}.$$
Now, if we write $\s'=(\l',\mu')$, with $\l' \in \SG_{[1,n-l]}$ 
and $\mu'\in\SG_{[n-l+1,n]}$, we have $\mu' \neq \s_{[n-l+1,n]}$ 
(because $\l \neq 1$). So, by Case 2, we have 
$$r_1 \cdots r_{n-l-1} r_n \s' \b^{\prime -1} \sim_L 
r_1 \cdots r_{n-l} \s' \b^{\prime -1}.$$
Therefore, $x'=w_n x \sim_L y'=w_ny$, and so $x \sim_L y$ 
by \cite[Corollary 11.7]{lusztig}.
\end{proof}

\bigskip

\begin{coro}\label{b sup}
Let $l \in \{1,\dots,n\}$, let $1 \le i_1 < \cdots < i_l \le n$, 
let $\s \in \SG_n$, let $\b \in Y_{l,n-l}$ and let $k \in [1,l]$ be 
such that $b \ge (i_k-1) a$. Then 
$$r_{i_1} \cdots r_{i_l} \s \sim_L r_1\cdots r_k r_{i_{k+1}} 
\cdots r_{i_l} \s.$$
\end{coro}

\bigskip

\begin{proof}
The proof proceeds essentially as in Lemma 
\ref{induction}. Let $w=r_{i_1} \cdots r_{i_l} \s$, let 
$w'=r_1\cdots r_k r_{i_{k+1}} \cdots r_{i_l} \s$ and let us 
write $w=vx^{-1}$ and $w' = v' x^{\prime -1}$ with 
$v$, $v' \in W_{i_k}$ and $x$, $x' \in X_n^{(i_k)}$. 

Since $w' w^{-1} = (r_1\cdots r_k)^{-1} (r_{i_1} \cdots r_{i_k}) \in W_{i_k}$, 
we get that $x=x'$. The same argument as in Lemma \ref{induction} 
shows that $v=r_{i_1} \cdots r_{i_k} \t$ and $v' = r_1 \cdots r_k \t'$ 
for some $\t$, $\t' \in \SG_{i_k}$. But 
$v' v^{-1}=w' w^{-1}=(r_1\cdots r_k)^{-1} (r_{i_1} \cdots r_{i_k})$, 
so $\t=\t'$. Now, by Proposition \ref{quasi asymptotique}, 
$v \sim_L v'$. So $w \sim_L w'$ by \cite[Proposition 9.11]{lusztig}.
\end{proof}

\bigskip

\section{Proof of Proposition \ref{exemples m2}\label{M2}}

\medskip

\begin{quotation}
{\bf Notation.} {\it In this section, and only in this section, we assume that 
$1 \le l \le n-1$ and that $(n-2)a < b \le (n-1) a$.} 
\end{quotation}

\medskip

We define a sequence $(C_j)_{l-1 \le j \le n-1}$ by induction as follows:
$$\begin{cases}
C_{l-1}=1, & \\
C_l = C_{s_l}, & \\
C_{j+1}=C_{s_{j+1}} C_j - C_{j-1}, & \text{if $l \le j \le n-2$.}
\end{cases}$$
Let $\mu$ denote the coefficient of $C_{a_{l-1}\s_{[l,n]}}$ in the expansion 
of $C_{n-1} C_{a_l \s_{[l,n]}}$ in the Kazhdan-Lusztig basis. 
To prove Proposition \ref{exemples m2}, it is sufficient to show 
the following statement:
\equat\label{mu}
\mu=\begin{cases}
1 & \text{if $b=(n-1)a$,}\\
Q^{-1} q^{n-1} + Q q^{1-n} & \text{if $(n-2) a < b < (n-1) a$.}
\end{cases}
\endequat

\begin{proof}[Proof of \ref{mu}]
If $r \in \ZM$, we set 
$$\HC_n[r]=\mathop{\oplus}_{\ell_t(w) \le r} A T_w = 
\mathop{\oplus}_{\ell_t(w) \le r} A C_w.$$
We shall show that 
\equat\label{?}
\begin{array}{c}
C_{n-1} C_{a_l \s_{[l,n]}} \equiv 
T_{s_{n-1}\cdots s_{l+1} s_l a_l\s_{[l,n]}} \hskip4cm\\ \hskip4cm
+ Q^{-1}q^{n-1} T_{a_{l-1}\s_{[l,n]}} 
\mod \Bigl(\HC_n[l-2] + \HC_n^{<0}\Bigr).
\end{array}
\endequat
The statement \ref{?} will be proved at the end of this section. 
Let us conclude the proof of \ref{mu}, assuming that \ref{?} holds. 

Let 
$$\mut=\begin{cases}
1 & \text{if $b=(n-1) a$},\\
Q^{-1} q^{n-1} + Q q^{1-n} & \text{if $(n-2) a < b < (n-1) a$.}
\end{cases}$$
We want to show that $\mu=\mut$. But, by \ref{?}, 
we have
$$C_{n-1} C_{a_l \s_{[l,n]}} - 
C_{s_{n-1}\cdots s_{l+1} s_l a_l\s_{[l,n]}} - \mut C_{a_{l-1}\s_{[l,n]}} 
\in \HC_n[l-2] + \HC_n^{<0} + \mathop{\oplus}_{w < a_{l-1}\s_{[l+1,n]}} 
A T_w.$$
Since 
$$\mathop{\oplus}_{w < a_{l-1}\s_{[l+1,n]}} A T_w
=\mathop{\oplus}_{w < a_{l-1}\s_{[l+1,n]}} A C_w,$$
there exists a family $(\nu_w)_{\ell_t(w) \le l-2
\text{ or }w < a_{l-1}\s_{[l+1,n]}}$ 
of elements of $A_{\ge 0}$ such that 
$$C_{n-1} C_{a_l \s_{[l,n]}} - 
C_{s_{n-1}\cdots s_{l+1} s_l a_l\s_{[l,n]}} - \mut C_{a_{l-1}\s_{[l,n]}} 
-\sum_{\stackrel{\ell_t(w) \le l-2}{\text{or }a_{l-1}\s_{[l+1,n]}}} 
\nu_w C_w \in \HC_n^{<0}.$$
Let $\nu_w'=\nu_w + \overline{\nu}_w - \t_A(\nu_w)$. Then 
$$C_{n-1} C_{a_l \s_{[l,n]}} - 
C_{s_{n-1}\cdots s_{l+1} s_l a_l\s_{[l,n]}} - \mut C_{a_{l-1}\s_{[l,n]}} 
-\sum_{\stackrel{\ell_t(w) \le l-2}{\text{or }a_{l-1}\s_{[l+1,n]}}} 
\nu_w' C_w \in \HC_n^{<0}$$
and $\overline{\nu}_w'=\nu_w$. So, if we set 
$$C=C_{n-1} C_{a_l \s_{[l,n]}} - 
C_{s_{n-1}\cdots s_{l+1} s_l a_l\s_{[l,n]}} - \mut C_{a_{l-1}\s_{[l,n]}} 
-\sum_{\stackrel{\ell_t(w) \le l-2}{\text{or }a_{l-1}\s_{[l+1,n]}}} 
\nu_w' C_w,$$
then 
$$\overline{C}=C\quad\text{and}\quad C \in \HC_n^{<0}.$$
So $C=0$ by \cite[Theorem 5.2]{lusztig}, and so $\mu=\mut$, as expected.
\end{proof}

\bigskip

So it remains to prove the statement \ref{?}:

\bigskip

\begin{proof}[Proof of \ref{?}]
First of all, we have 
$C_{a_l \s_{[l,n]}} = C_{a_l} C_{\s_{[l,n]}}$, since the supports 
of $a_l$ and $\s_{[l,n]}$ (in $S_n$) are disjoint. Moreover, 
since $l \le n-1$ (i.e. $a_l \in W_{n-1}$) 
and $b > (n-2) a$, it follows from 
\cite[Propositions 2.5 and 5.1]{bilatere} that 
\equat\label{cal}
C_{a_l}=(T_{t_1} + Q^{-1})(T_{t_2}+Q^{-1}) \cdots (T_{t_l}+Q^{-1}) 
T_{\s_l}^{-1}.
\endequat
Let $\HC(\SG_n)$ denote the sub-$A$-algebra of $\HC_n$ generated 
by $T_{s_1}$,\dots, $T_{s_{n-1}}$. It is the Hecke algebra of $\SG_n$ 
(with parameter $a$). Then $\HC_n[l-2]$ is a sub-$A$-module of $\HC_n$. 
Therefore, it follows from \ref{cal} that 
$$C_{a_l} \equiv \bigl(T_{w_l} + Q^{-1} \sum_{1 \le i \le l} 
T_{t_1\cdots t_{i-1} t_{i+1}\cdots t_l}\bigr) T_{\s_l}^{-1} \mod \HC_n[l-2].$$
But, if $1 \le i \le l$, then
$$t_1 \cdots t_{i-1} t_{i+1}\cdots t_l = 
s_i s_{i+1} \cdots s_{l-1} a_{l-1} \s_{l-1} s_{l-1} \cdots s_{i+1} s_i,$$
and $\s_l=  s_{l+1-i} \cdots s_{l-2} s_{l-1} 
\s_{l-1} s_{l-1} \cdots s_{i+1} s_i$. Moreover, 
$$\ell(\s_l)=\ell(s_{l+1-i} \cdots s_{l-2} s_{l-1}) + 
\ell(\s_{l-1} s_{l-1} \cdots s_{i+1} s_i).$$
Therefore, 
$$C_{a_l} \equiv T_{a_l} + Q^{-1} \sum_{1 \le i \le l}
T_{s_i s_{i+1}\cdots s_{l-1}} T_{a_{l-1}} 
(T_{s_{l+1-i}\cdots s_{l-2}s_{l-1}})^{-1} \mod \HC_n[l-2].$$
Finally, we get
$$C_{a_l \s_{[l,n]}} \equiv T_{a_l} C_{\s_{[l,n]}} + Q^{-1} \sum_{1 \le i \le l}
T_{c_{[i,l-1]}} T_{a_{l-1}} 
(T_{c_{[l+1-i,l-1]}})^{-1} C_{\s_{[l,n]}} \mod \HC_n[l-2].$$
Now, if $l-1 \le j \le n-1$, then 
\equat\label{congruence}
\begin{array}{c}
C_j C_{a_l \s_{[l,n]}} \equiv \DS{\sum_{i=l-1}^j} q^{i-j} T_{d_{[l,i]}} 
T_{a_l} C_{\s_{[l,n]}} \hskip4cm \\ \hskip1cm
+ Q^{-1} C_j \DS{\sum_{1 \le i \le l}}
T_{c_{[i,l-1]}} T_{a_{l-1}} 
(T_{c_{[l+1-i,l-1]}})^{-1} C_{\s_{[l,n]}}\mod \HC_n[l-2].
\end{array}
\endequat
\begin{quotation}
\begin{proof}[Proof of \ref{congruence}]
We shall argue by induction on $j$. The cases where $j=l-1$ or 
$j=l$ are obvious. So assume that $j \in [l,n-2]$ and that 
\ref{congruence} holds for $j$. By the induction hypothesis, we get 
$$\begin{array}{c}
C_{j+1} C_{a_l \s_{[l,n]}} \equiv C_{s_{j+1}} \DS{\sum_{i=l-1}^j} q^{i-j} T_{d_{[l,i]}} 
T_{a_l} C_{\s_{[l,n]}} - \DS{\sum_{i=l-1}^{j-1}} q^{i-j+1} T_{d_{[l,i]}} 
T_{a_l} C_{\s_{[l,n]}} \\ 
+ Q^{-1} C_{j+1} \DS{\sum_{1 \le i \le l}}
T_{c_{[i,l-1]}} T_{a_{l-1}} 
(T_{c_{[l+1-i,l-1]}})^{-1} C_{\s_{[l,n]}}\mod \HC_n[l-2].
\end{array}$$
Now, 
$$C_{s_{j+1}} T_{d_{[l,j]}} T_{a_l} C_{\s_{[l,n]}} = 
T_{d_{[l,j+1]}} T_{a_l} C_{\s_{[l,n]}} + 
q^{-1} T_{d_{[l,j]}} T_{a_l} C_{\s_{[l,n]}}$$
and, if $l-1 \le i < j$, then 
$$C_{s_{j+1}} T_{d_{[l,i]}} T_{a_l} C_{\s_{[l,n]}} = 
T_{d_{[l,i]}} T_{a_l} C_{s_{j+1}} C_{\s_{[l,n]}} = 
(q+q^{-1}) T_{d_{[l,i]}} T_{a_l} C_{\s_{[l,n]}}.$$
Now \ref{congruence} follows from a straightforward computation.
\end{proof}
\end{quotation}
\bigskip

Since $d_{[l,i]} \in Y_{l,n-l}$, we have 
$$T_{d_{[l,i]}} T_{a_l} C_{\s_{[l,n]}} = T_{d_{[l,i]} a_l} C_{\s_{[l,n]}} 
\equiv T_{d_{[l,i]} a_l\s_{[l,n]}} \mod \HC_n^{<0},$$
so, by \ref{congruence}, we get
\begin{multline*}
C_{n-1} C_{a_l \s_{[l,n]}} \equiv T_{d_{[l,n-1]} a_l\s_{[l,n]}} \hskip6cm \\
+ Q^{-1} C_{n-1} \DS{\sum_{1 \le i \le l}}
T_{c_{[i,l-1]}} T_{a_{l-1}} 
(T_{c_{[l+1-i,l-1]}})^{-1} C_{\s_{[l,n]}}\mod 
\Bigl(\HC_n[l-2]+\HC_n^{<0}\Bigr). 
\end{multline*}
For $1 \le i \le l$, let $\XC_i=Q^{-1} C_{n-1} T_{c_{[i,l-1]}} T_{a_{l-1}} 
(T_{c_{[l+1-i,l-1]}})^{-1} C_{\s_{[l,n]}}$.
There exists a family $(f_I)_{I \subseteq [l,n-1]}$ of elements 
of $\ZM$ such that $C_{n-1}=\sum_{I \subseteq [l,n-1]} f_I C_{d_I}$. 
Moreover, $f_{[l,n-1]}=1$. Also,
$$(T_{c_{[l+1-i,l-1]}})^{-1} = \sum_{J \subseteq [l+1-i,l-1]} 
(q-q^{-1})^{i-1-|J|} T_{c_J}.$$
Therefore, 
$$\XC_i=\sum_{\substack{I \subseteq [l,n-1] \\ J \subseteq [l+1-i,l-1]}}
f_I Q^{-1} (q-q^{-1})^{i-1-|J|} C_{d_I} T_{c_{[i,l-1]}} T_{a_{l-1}} 
T_{c_J} C_{\s_{[l,n]}}.$$
Let $\D_{i,I,J}=f_I Q^{-1} (q-q^{-1})^{i-1-|J|} C_{d_I} 
T_{c_{[i,l-1]}} T_{a_{l-1}} T_{c_J} C_{\s_{[l,n]}}$. If we express 
$\D_{i,I,J}$ in the standard basis $(T_w)_{w \in W_n}$, then 
the degree of the coefficients are bounded by 
$-b + (i-1-|J| + |I|) a$. Since $b > (n-2) a$, this degree is 
in $\G_{<0}$, except if $i=l$, $J=\vide$ and $I=[l,n-1]$. Therefore, 
$$C_{n-1} C_{a_l \s_{[l,n]}} \equiv T_{d_{[l,n-1]} a_l\s_{[l,n]}} 
+ \D_{l,[l,n-1],\vide} \mod \Bigl(\HC_n[l-2]+\HC_n^{<0}\Bigr).$$
But
\eqna
\D_{l,[l,n-1],\vide} &=& Q^{-1} (q-q^{-1})^{l-1} 
C_{s_{n-1}\cdots s_l} T_{a_{l-1}} C_{\s_{[l,n]}}\\
&=& Q^{-1} (q-q^{-1})^{l-1} T_{a_{l-1}}C_{s_{n-1}\cdots s_l} C_{\s_{[l,n]}}\\
&=& Q^{-1} (q-q^{-1})^{l-1}(q+q^{-1})^{n-l} T_{a_{l-1}} C_{\s_{[l,n]}},
\endeqna
the last equality following from Theorem \ref{kl plus} (a). 
So $\D_{l,[l,n-1],\vide} \equiv Q^{-1} q^{n-1} T_{a_{l-1}} C_{\s_{[l,n]}} 
\mod \HC_n^{<0}$. The proof of \ref{?} is complete.
\end{proof}

\bigskip

\section{Consequences of Proposition \ref{exemples m2}}

\medskip

The aim of this section is to prove the following proposition:

\bigskip

\begin{prop}\label{coro exemples m2}
Let $l \in \{1,\dots,n\}$ and assume that $b \le (n-1) a$. Then 
$$s_1s_2\cdots s_{n-1} a_{l-1}\s_{[l,n-1]} \sim_L 
ts_1s_2\cdots s_{n-1} a_{l-1}\s_{[l,n-1]}.$$
\end{prop}

\begin{proof}
Let $u_{l,n}=ts_1s_2\cdots s_{n-1} a_{l-1}\s_{[l,n-1]}=
ts_1\cdots s_{l-1}a_{l-1} s_l\cdots s_{n-1}\s_{[l,n-1]}= a_l \s_{[l,n]}$. 
We need to show that $tu_{l,n} \sim_L u_{l,n}$ 
(note that $tu_{l,n} \le u_{l,n}$). We shall argue by induction on $n$, 
the cases where $n=1$ or $2$ being obvious. So assume that $n \ge 3$ and 
that $tu_{l,n-1} \sim_L u_{l,n-1}$ if $b \le (n-2) a$. 

First, assume that $b \le (n-2) a$. Then 
$$\begin{cases}
u_{l,n}=u_{l,n-1} s_{n-1} \cdots s_{l+1} s_l & \text{if $l \le n-1$},\\
u_{l,n}=a_n=u_{l-1,n-1} s_{n-1} \cdots s_2 s_1 t & \text{if $l=n$.} 
\end{cases}$$
By the induction hypothesis, we have $tu_{k,n-1} \sim_L u_{k,n-1}$ so, 
since $s_{n-1} \cdots s_{l+1} s_l$ and $s_{n-1} \cdots s_2 s_1 t$ belong 
to $X_n^{-1}$, it follows from \cite[Proposition 9.11]{lusztig} 
that $tu_{l,n} \sim_L u_{l,n}$. 

This means that we may, and we will, assume that $(n-2)a < b \le (n-1)a$.
But, by Proposition \ref{exemples m2}, we have 
$a_{l-1} \s_{[l,n]} \pre{L} a_l \s_{[l,n]} = u_{l,n}$. On the other hand,
$$tu_{l,n} = c_{[1,l-1]} a_{l-1} \s_{[l,n]} 
\pre{L} c_{[2,l-1]} a_{l-1} \s_{[l,n]} \pre{L} \cdots \pre{L} 
s_{l-1} a_{l-1} \s_{[l,n]}\pre{L} a_{l-1} \s_{[l,n]}.$$
So $tu_{l,n} \sim_L u_{l,n}$, as desired.
\end{proof}

\bigskip

\remark{if and only if}
Note that the converse of Proposition \ref{coro exemples m2} 
also holds. Indeed, if $b > (n-1) a$ and if $x \sim_L y$ 
for some $x$ and $y$ in $W_n$, then $\ell_t(x)=\ell_t(y)$ 
(see \cite[Theorem 7.7]{lacriced} and \cite[Corollary 5.2]{bilatere}).\finl 

\bigskip

\begin{coro}\label{fin du coup}
Let $l \in \{1,2,\dots,n\}$ and let $\b \in Y_{l-1,n-l}$. Then 
$$s_1s_2\cdots s_{n-1} a_{l-1}\s_{[l,n-1]}\b^{-1} \sim_L 
ts_1s_2\cdots s_{n-1} a_{l-1}\s_{[l,n-1]}\b^{-1}.$$
\end{coro}

\bigskip

\begin{proof}
Let $w=s_1s_2 \cdots s_{n-1} a_l \s_{[l+1,n-1]}\b^{-1}$. 
We want to show that $w \sim_L tw$. We shall argue by induction on $\ell(\b)$. 
If $\ell(\b)=0$ (i.e. $\b=1$), this is just the proposition 
\ref{coro exemples m2}. Sp we assume now that $\ell(\b) \ge 1$. 
We shall use the $*$-operation (see \S\ref{star sec}). 
For this, we need to study the action of the $\g_i$'s on $w$, when 
possible. 

We have $\s_{[l,n-1]}a_{l-1}=a_{l-1} \s_{[l,n-1]}$, so 
$$w=s_1s_2 \cdots s_{n-1} \s_{[l,n-1]} (a_{l-1}\b)^{-1}=
s_1s_2\cdots s_{l-1} \s_{[l,n]} (a_{l-1}\b)^{-1}.$$ 
Let $1 \le j_1 < \cdots < j_{l-1} \le n-1$ be the unique sequence 
such that $a_l\b=r_{j_1} r_{j_2} \cdots r_{j_{l-1}}$. Since 
$\ell(\b) > 0$, we have $(j_1,j_2,\dots,j_{l-1}) \neq (1,2,\dots,l-1)$, 
so there exists $k \in [1,l-1]$ such that $j_k-j_{k-1} \ge 2$ 
(where $j_0 = 0$ by convention). Note that $j_k < n$ so 
$j_k+1 \in [2,n]$. We have, by \ref{ri1 ril} 
\eqna
w(j_k) s_1\cdots s_{l-1} \s_{[l,n]} 
(r_{j_1}\cdots r_{j_{l-1}})^{-1}(j_k)&=&s_1\cdots s_{l-1} \s_{[l,n]}(k-l)\\
&=&-s_1\cdots s_{l-1} (l-k) \\
&=& -(l+1-k) < 0
\endeqna
and
\eqna
w(j_k-1) s_1\cdots s_{l-1} \s_{[l,n]} (r_{j_1}\cdots r_{j_{l-1}})^{-1}(j_k-1)
&=&s_1\cdots s_{l-1} \s_{[l,n]}(l+q)\\
&=&s_1\cdots s_{l-1} (n+1-q) \\ &=& n+1-q > 0
\endeqna
for some $q \in [1,n+1-l]$. Moreover,
a similar computation shows that (with the convention that $j_{l}=n+1$) 
$$w(j_k)=
\begin{cases}
-(l-k) & \text{if $j_{k+1} =j_k+1$,}\\
n-q & \text{if $j_{k+1} \ge j_k + 2$.}
\end{cases}$$
In any case, we have
$$w(j_k) < w(j_k+1) < w(j_k-1).$$
This shows that
$$w s_{j_k-1} s_{j_k} < w s_{j_k-1} < w < w s_{j_k},$$
So $w \in \DC_{j_k-1}(W_n)$ and $\g_{j_k-1}(w) = w s_{j_k-1} < w$. 
Now, let $\b'=s_{j_k}\b$. An easy computation as above shows that 
$\b' < \b$, so that $\b' \in Y_{l,n-1-l}$ by Deodhar's Lemma 
(see \cite[Lemma 2.1.2]{ourbuch}). So $\g_{j_k}(w)=s_1\cdots s_{n-1} 
a_{l-1} \s_{[l,n-1]} \b^{\prime -1}$ where $\b' \in Y_{l,n-1-l}$ 
is such that $\ell(\b') = \ell(\b)-1$. But, by Corollary \ref{star t}, 
we have $t\g_i(w)=\g_i(tw)$. So , by \ref{star} and by the induction 
hypothesis, we get that $w \sim_L tw$, as desired.
\end{proof}

\bigskip

\section{Proof of Theorem \ref{main}}

\medskip

\soussection{Knuth relations} 
By recent results of Taskin \cite[Theorems 1.2 and 1.3]{taskin}, 
the equivalence relations $\sim_R^r$ and $\simeq_R^r$ can be described using generalisations 
of Knuth relations (for the relation $\simeq_R^r$, a similar result has been obtained 
independently by Pietraho \cite[Theorems 3.8 and 3.9]{pietraho} using other kinds of Knuth 
relations). We shall recall here Taskin's construction. For this, we shall need the following 
notation: if $0 \le r \le n-2$, we denote by $\EC_n^{(r)}$ the set of elements $w \in W_n$ such that 
$|w(1)| > |w(i)|$ for $i \in \{2,3,\dots,r+2\}$ and such that the sequence 
$(w(2),w(3),\dots,w(r+2))$ is a shuffle of a positive decreasing sequence 
and a negative increasing sequence. If $r \ge n-1$, we set $\EC_n^{(r)}=\vide$.
Following \cite[Definition 1.1]{taskin}, we introduce three relations which 
will be used to generate the relations $\sim_R^r$ and $\simeq_R^r$.

Let $w$, $w' \in W_n$ and let $r \ge 0$:

\begin{itemize}
\item[$\bullet$] We write $w \smile_1 w'$ if there exists $i \ge 2$ (respectively $i \le n-2$) 
such that $w(i) < w(i-1) < w(i+1)$ (respectively $w(i) < w(i+2) < w(i+1)$) and 
$w'=ws_i$.

\item[$\bullet$] We write $w \smile_2^r w'$ if there exists $i \le {\mathrm{min}}(r,n-1)$ such that 
$w(i)w(i+1) < 0$ and $w'=ws_i$. The relation $\smile_2^0$ never 
occurs.

\item[$\bullet$] We write $w \smile_3^r w'$ if $w \in \EC_n^{(r)}$ and $w'=wt$. 
If $r \ge n-1$, the relation $\smile_3^r$ never occurs.
\end{itemize}

\bigskip

\noindent{\sc Remark - } If $w \smile_2^r w'$, then $w \smile_2^{r+1} w'$. If $w \smile_3^r w'$, 
then $w \smile_3^{r-1} w'$ (indeed, $\EC_n^{(r)} \subseteq \EC_n^{(r-1)}$).\finl

\bigskip

\noindent{\bf Taskin's Theorem.} 
{\it With the above notation, we have:
\begin{itemize}
\itemth{a} The relation $\sim_R^r$ is the equivalence relation generated by 
the relations $\smile_1$, $\smile_2^r$ and $\smile_3^r$.

\itemth{b} The relation $\simeq_R^r$ is the equivalence relation generated by 
the relations $\smile_1$, $\smile_2^r$ and $\smile_3^{r-1}$.
\end{itemize}}

\bigskip

\soussection{Proof of Theorem \ref{main}}
Recall that the relation $\sim_{LR}^r$ (respectively $\simeq_{LR}^r$) is the equivalence relation 
generated by $\sim_L^r$ and $\sim_R^r$ (respectively $\simeq_L^r$ and $\simeq_R^r$). Recall also 
that $x \sim_L^r y$ (respectively $x \simeq_L^r y$, respectively $x \sim_L y$) if 
and only if $x^{-1} \sim_R^r y^{-1}$ (respectively $x^{-1} \simeq_R^r y^{-1}$, respectively 
$x^{-1} \sim_R y^{-1}$). So it is sufficient to show that Theorem \ref{main} holds whenever 
$? = R$. It is then easy to see that Theorem \ref{main} will follow from Taskin's Theorem 
and from the following three lemmas (which will be proved in subsections \ref{sub 1}, 
\ref{sub 2} and \ref{sub 3}).

\bigskip

\begin{lem}\label{1}
Let $w$, $w' \in W_n$ be such that $w \smile_1 w'$. 
Then $w \sim_R w'$.
\end{lem}

\bigskip

\begin{lem}\label{2}
Let $w$, $w' \in W_n$ and let $r \ge 0$ be such that 
$b \ge ra$ and $w \smile_2^r w'$. Then $w \sim_R w'$.
\end{lem}

\bigskip

\begin{lem}\label{3}
Let $w \in W_n$ and let $r \ge 0$ be such that 
$b \le (r+1)a$ and $w \smile_3^r w'$. Then $w \sim_R w'$.
\end{lem}

\bigskip

\soussection{Proof of Lemma \ref{1}\label{sub 1}}
Let $w$, $w' \in W_n$ be such that $w \smile_1 w'$. Let $i \in I_{n-1}^+$ be such that $w'=ws_i$. 
Then $i \ge 2$ and $w(i) < w(i-1) < w(i+1)$, or $i \le n-2$ and $w(i) < w(i+2) < w(i+1)$. 
In the first case, we have $ws_is_{i-1} > ws_i > w > ws_{i-1}$ while, in the second case, 
we have $ws_is_{i+1} > ws_i > w > ws_{i+1}$. So $w'=ws_i \sim_R w$ by 
\ref{siml simr} and Corollary \ref{kl}. 
The proof of Lemma \ref{1} is complete.

\bigskip

\soussection{Proof of Lemma \ref{2}\label{sub 2}}
Let $w$, $w' \in W_n$ and let $r \ge 0$ be such that 
$b \ge ra$ and $w \smile_2^r w'$. Let $i \in I_{n-1}^+$ be the element such that 
$w'=ws_i$. Then $i \le r$ and $w(i)w(i+1) < 0$. 
By exchanging $w$ and $w'$ if necessary, we may assume that $w(i) < 0$ and $w(i+1) > 0$. 

Let us write $w=xv$, with $x \in X_n^{(i+1)}$ and $v \in W_{i+1}$. 
Then $vs_i \in W_{i+1}$ and $ws_i=xvs_i$. 
Therefore, by \cite[Proposition 9.11]{lusztig}, we only need to show that 
$vs_i \sim_L v$. But $0 < x(1) < \cdots < x(r+1)$ (see Lemma \ref{caracterisation X}), 
and $v(j) \in I_{i+1}$ for all $j \in I_{i+1}$. So $v(i) > 0$ and $v'(i+1) < 0$. 
In particular, $v \smile_2^r vs_i$ (and even $v \smile_2^i v'$). 
This means that we may (and we will) assume that $i=n-1$. So we have
$$b \ge (n-1) a,\quad w(n-1) < 0 \quad\text{and}\quad w(n) > 0,$$
and we want to show that $w \sim_R ws_{n-1}$ or, in other words, that 
$$w^{-1} \sim_L s_{n-1}w^{-1}.\leqno{(?)}$$

Let $\a=\a_{w^{-1}}$, $\s=\s_{w^{-1}}$ and $\b=\b_{w^{-1}}$. Then
$$w^{-1}=\a a_l \s a_l \b^{-1}.$$
By Lemma \ref{decomposition}, there exists a unique sequence 
$1 \le i_1 < \cdots < i_l \le n$ such that 
$\a a_l = r_{i_1}\cdots r_{i_l}$ so 
$$w^{-1}=r_{i_1} \cdots r_{i_l} \s \b^{-1}.$$
But, again by Lemma \ref{decomposition}, we have $w^{-1}(i) < 0$ 
if and only if $i \in \{i_1,\dots,i_l\}$. So 
$$i_l=n-1.$$
So 
$$w^{-1}=r_{i_1} \cdots r_{i_{l-1}}r_{n-1} \s \b^{-1}.$$
and
$$s_{n-1}w^{-1}=r_{i_1} \cdots r_{i_{l-1}}r_n \s \b^{-1}.$$
So the result follows from Proposition \ref{quasi asymptotique}.

\bigskip

\soussection{Proof of Lemma \ref{3}\label{sub 3}} 
Let $w \in W_n$ and let $r \ge 0$ be such that 
$b \le (r+1)a$ and $w^{-1} \smile_3^r w^{\prime -1}$. We want to show that $w \sim_L w'=tw$. 
The proof goes through several steps.

\medskip

\noindent{\it First step: easy reductions.} 
First, note that $r \le n-2$. 
Let us write $w=vx^{-1}$, with $v \in W_{r+2}$ and $x \in X_n^{(r+2)}$. 
Then $0 < x(1) < \cdots < x(r+2)$ by Lemma \ref{caracterisation X}, 
so $v^{-1} \in \EC_{r+2}^{(r)}$. 
Then $tw=(tv) x^{-1}$ with $tv \in W_{r+2}$ so, by 
\cite[Proposition 9.11]{lusztig}, it is sufficient to show 
that $tv \sim_L v$. This shows that we may (and we will) assume that $r=n-2$. 
%In particular, $|w^{-1}(k)| = n+1-k$ for all $k \in [1,n]$. 

\smallskip

By \cite[Corollary 11.7]{lusztig}, this is equivalent to show 
that $tw_nw \sim_L w_n w$. Since $w_nw \in \EC_n^{(n-2)}$ we may, 
by replacing $w$ by $tw$, $w_nw$ or $tw_n w$, 
assume that $w^{-1}(1) > 0$ and $w^{-1}(n) > 0$. 
Since moreover $|w^{-1}(1)| > |w^{-1}(i)|$ for all $i \in \{2,3,\dots,n=r+2\}$, 
we have $w^{-1}(1)=n$. 

\smallskip

As a conclusion, we are now working under the following hypothesis:

\medskip

\begin{quotation}
{\bf Hypothesis.} 
{\it From now on, and until the end of this subsection, 
we assume that
\begin{itemize}
\itemth{1} $w^{-1}(1)=n$ and $w^{-1}(n) > 0$, and 

\itemth{2} $w^{-1} \in \EC_n^{(n-2)}$. 
\end{itemize}}
\end{quotation}

\medskip

And recall that we want to show that 
$$tw \sim_L w .\leqno{(?)}$$

\bigskip

\noindent{\it Second step: decomposition of $w$.} 
Let $v=s_{n-1}\cdots s_2 s_1 w$. Then $v^{-1}(n)=w^{-1}(1)=n$ by (3), 
so $v \in W_{n-1}$. Therefore, 
\equat\label{dec w}
w=s_1 s_2 \cdots s_{n-1} v,\quad 
s_1s_2 \cdots s_{n-1} \in X_n\quad\text{and}\quad
v \in W_{n-1}.
\endequat
Note that 
\equat\label{valeur v}
v^{-1}(k)=w^{-1}(k+1)
\endequat
for all $k \in [1,n-1]$, so that 
\equat\label{ordre v}
v \in \EC_{n-1}^{(n-3)}
\endequat
and, by (2), 
\equat\label{v positif}
v^{-1}(n-1) > 0.
\endequat
Let us write $v=r_{i_1} \cdots r_{i_l} \s \b^{-1}$, 
with $l=\ell_t(v)=\ell_t(w)$, $1 \le i_1 < \cdots < i_l \le n-1$, 
$\s \in \SG_{l,n-1-l}$ and $\b \in Y_{l,n-1-l}$. By 
\ref{v positif} and Lemma \ref{decomposition}, we have
\equat\label{il inf}
i_l \le n-2.
\endequat
Finally, note that 
\equat\label{sigma}
\s=\s_{[l+1,n-1]}.
\endequat
\begin{proof}[Proof of \ref{sigma}]
By \ref{ordre v}, we have $|v^{-1}(i_1)| > |v^{-1}(i_2)| > \cdots > |v^{-1}(i_l)|$. 
Therefore, it follows from \ref{ri1 ril} that 
$\b(\s^{-1}(l)) > \b(\s^{-1}(l-1)) > \cdots > \b(\s^{-1}(1))$. 
Since $\s$ stabilizes the interval $[1,l]$ and since 
$\b$ is increasing on $[1,l]$ (because it lies in $Y_{l,n-l}$), 
this forces $\s(k)=k$ for all $k \in [1,l]$. 

Similarly, if $1 \le j_1 < \cdots < j_{n-l} \le n$ denotes 
the unique sequence such that $[1,n]=\{i_1,\dots,i_l\} 
\cup \{j_1,\dots,j_{n-l}\}$, then 
$|v^{-1}(j_1)| > |v^{-1}(j_2)| > \cdots > |v^{-1}(j_{n-l})|$ by \ref{ordre v}. 
So it follows from \ref{ri1 ril} that 
$\b(\s^{-1}(l+1)) > \b(\s^{-1}(l+2)) > \cdots > \b(\s^{-1}(n))$ 
and, since $\s$ stabilizes the interval $[l+1,n]$ and $\b$ is 
increasing on the same interval, this forces $\s(l+k)=n+1-k$ 
for $k \in [1,n-l]$. 
\end{proof}

\bigskip

\noindent{\it Third step: conclusion.}
We first need the following elementary result:
\equat\label{s1 sn}
s_1 s_2 \cdots s_{n-1} r_{i_1} \cdots r_{i_l} = r_{i_1+1} \cdots r_{i_l+1} 
s_{l+1} s_{l+2} \cdots s_{n-1}.
\endequat
\begin{proof}[Proof of \ref{s1 sn}]
This follows easily from \ref{ri} or from \ref{sk rl}. 
\end{proof}

\bigskip
Now, let $\t=s_{l+1}s_{l+2} \cdots s_{n-1} \s_{[l+1,n]} \b^{-1} = 
\s_{[l+1,n]} \b^{-1}\in \SG_n$. 
Then, by \ref{s1 sn}, we have 
$$w=r_{i_1+1} r_{i_2+1} \cdots r_{i_l+1} \t\quad\text{and}\quad 
tw = r_1 r_{i_1+1} r_{i_2+1} \cdots r_{i_l+1} \t.$$
By \ref{il inf}, we have $b \ge (i_l+1 -1) a$, so, by 
Corollary \ref{b sup}, we have 
$$w \sim_L r_2 r_3 \cdots r_{l+1} \t\quad\text{and}\quad
tw \sim_L r_1 r_2 \cdots r_{l+1} \t.$$
So we only need to show that 
$r_2 r_3 \cdots r_{l+1} \t \sim_L r_1 r_2 \cdots r_{l+1} \t=tr_2 r_3 \cdots r_{l+1} \t$. 
But $r_2 \cdots r_{l+1} \s_{[l+1,n]} \b^{-1}s_1\cdots s_{n-1} a_l \s_{[l+1,n-1]} \b^{-1}$, 
with $\b \in Y_{l,n-1-l}$. So the result follows from Corollary \ref{fin du coup}.

The proof of Lemma \ref{3} is complete, as well as the proof of Theorem \ref{main}.

\end{document}